\documentclass[11pt]{elsarticle}
\usepackage{amsmath}
\usepackage{amssymb}
\usepackage{amsfonts}
\usepackage{mathrsfs}
\usepackage{graphicx}
\usepackage{epsfig}
\usepackage{stmaryrd} 
\usepackage[english, frenchb]{babel}
\newtheorem{theo}{Theorem}
\newtheorem{prop}[theo]{Proposition}

\newtheorem{coro}[theo]{Corollary}
\newtheorem{lemm}[theo]{Lemma}
\newtheorem{Main Theorem}{Main Theorem}
\newtheorem{theoalph}{Theorem}

\newdefinition{defi}[theo]{Definition}

\newdefinition{rema}[theo]{Remark}
\newdefinition{exam}[theo]{Example}
\newdefinition{ques}[theo]{Question}

\newproof{proof}{Proof}

\newcommand{\partn}[1]{{\smallskip \noindent \textbf{#1.}}}

\numberwithin{equation}{section}       

\DeclareMathAlphabet{\mathpzc}{OT1}{pzc}{m}{it} 

\newcommand{\C}{\mathbb{C}}

\newcommand{\N}{\mathbb{N}}

\newcommand{\Q}{\mathbb{Q}}
\newcommand{\R}{\mathbb{R}}

\newcommand{\Z}{\mathbb{Z}}

\newcommand{\cM}{\mathcal{M}}

\newcommand{\cO}{\mathcal{O}}

\newcommand{\sC}{\mathscr{C}}

\newcommand{\sP}{\mathscr{P}}

\newcommand{\tA}{\widetilde{A}}

\newcommand{\teta}{\widetilde{\teta}}

\renewcommand{\=}{ : = }

\newcommand{\parameter}{\lambda}
\newcommand{\expansion}[1]{\langle #1 \rangle}
\newcommand{\pcs}{post\nobreakdash-critical set}

\newcommand{\Choquet}{\mathscr{C}}
\newcommand{\Polish}{\sP}
\newcommand{\Sunimodal}{\textrm{S}\nobreakdash-unimodal}

\newcommand{\lefti}{I}
\newcommand{\righti}{J}
\newcommand{\kneading}{Q_{(a, q)}}

\begin{document}

\title{Choquet simplices as spaces of invariant probability measures on post-critical sets}

\author[mi]{Mar{\'\i}a Isabel Cortez\fnref{fn1}}
\ead{maria.cortez@usach.cl}
\address[mi]{Departamento de Matem{\'a}tica y Ciencia de la Computaci{\'o}n, Universidad de Santiago de Chile, Av. Libertador Bernardo O'Higgins~3363, Santiago, Chile.}
\fntext[fn1]{Partially supported by Fondecyt de Iniciaci{\'o}n 11060002, Nucleus Millenius P04-069-F, and Research Network on Low Dimensional Dynamics, PBCT ACT-17, CONICYT, Chile} 

\author[j]{Juan Rivera-Letelier\fnref{fn2}}
\ead{riveraletelier@mat.puc.cl}
\address[j]{Facultad de Matem{\'a}ticas, Campus San Joaqu{\'\i}n, P.~Universidad Cat{\'o}lica de Chile, Avenida Vicu{\~n}a Mackenna~4860, Santiago, Chile}
\fntext[fn2]{Partially supported by Research Network on Low Dimensional Dynamics, PBCT ACT-17, CONICYT, Chile} 

\begin{abstract}
A well-known consequence of the ergodic decomposition theorem is that the space of invariant probability measures of a topological dynamical system, endowed with the weak$^*$ topology, is a non-empty metrizable Choquet simplex.
We show that every non-empty metrizable Choquet simplex arises as the space of invariant probability measures on the \pcs{} of a logistic map.
Here, the \pcs{} of a logistic map is the $\omega$-limit set of its unique critical point.
In fact we show the logistic map~$f$ can be taken in such a way that its \pcs{} is a Cantor set where~$f$ is minimal, and such that each invariant probability measure on this set has zero Lyapunov exponent, and is an equilibrium state for the potential~$- \ln |f'|$.
\end{abstract}



\maketitle

%
%

\section{Introduction}
A well-known consequence of the ergodic decomposition theorem is that the space of invariant probability measures of a topological dynamical system, endowed with the weak$^*$ topology, is a metrizable Choquet simplex.\footnote{See for example~\cite[p.~95]{Gla03}.
We recall the definition of Choquet simplex in~\S\ref{ss:Choquet simplices}.}
The purpose of this paper is to show that every non-empty metrizable Choquet simplex arises in this way within the \textit{logistic family} of maps~$(f_\parameter)_{(0, 4]}$, where for each parameter~$\parameter$ in $(0, 4]$ the \textit{logistic map}~$f_\parameter : [0, 1] \to [0, 1]$ is defined by,
$$ f_\parameter(x) = \parameter x (1 - x). $$
To make a more precise statement note that $x = \tfrac{1}{2}$ is the unique point in $[0, 1]$ at which the derivative of~$f_{\parameter}$ vanishes.
We call $x = \tfrac{1}{2}$ the \textit{critical point} of~$f_\parameter$, and its $\omega$\nobreakdash-limit set is called the \textit{\pcs{}} of~$f_{\parameter}$.
It is a compact set that is forward invariant by~$f_{\parameter}$.

The following is our main result.
Recall that for a compact topological space~$X$ a continuous map~$T : X \to X$ is \textit{minimal}, if every forward orbit of~$T$ is dense in~$X$.

\

\noindent
\textbf{Main Theorem.} \textit{For each non-empty metrizable Choquet simplex~$\Choquet$ there is a parameter~$\parameter \in (0, 4]$ such that the \pcs{} of the logistic map~$f_{\parameter}$ is a Cantor set, the restriction of~$f_{\parameter}$ to this set is minimal, and such that the space of invariant probability measures supported by this set, endowed with weak$^*$ topology, is affine homeomorphic to~$\Choquet$.}

\

The first result of this kind was shown by Downarowicz in~\cite{Dow91}, who showed that every non-empty metrizable Choquet simplex arises, up to an affine homeomorphism, as the space of invariant probability measures of a ``minimal Cantor system''; that is, a dynamical system generated by a minimal homeomorphism of a Cantor set.
In fact he showed that the minimal Cantor system can be taken as a ``0\nobreakdash-1 Toeplitz flow'': A special type of subshift of~$\{0, 1 \}^{\Z}$.
See also~\cite{GjeJoh00,Orm97} for a different approach to this result, and see~\cite{Cor06} for an analogous result in the case of actions of ~$\Z^d$.

The following corollary is a direct consequence of the Main Theorem and of the fact that for each non-empty Polish space~$\Polish$ there is a metrizable Choquet simplex whose set of extreme points is homeomorphic to~$\Polish$, see for example~\cite{Hay75}.
\begin{coro}\label{c:extreme realization}
For each non-empty Polish space~$\Polish$ there is a parameter~$\parameter \in (0, 4]$ such that the \pcs{} of the logistic map~$f_{\parameter}$ is a Cantor set, the restriction of~$f_{\parameter}$ to this set is minimal, and such that the space of ergodic and invariant probability measures supported by this set, endowed with weak$^*$ topology, is homeomorphic to~$\Polish$.
\end{coro}
The special case where the Polish space~$\Polish$ is compact and totally disconnected is precisely~\cite[Main Theorem]{CorRiv0804}.
The first result in this direction was shown by~Bruin, who gave an example of a parameter~$\parameter \in (0, 4]$ such that the \pcs{} of~$f_\parameter$ is a Cantor set where~$f_\parameter$ is minimal, but not uniquely ergodic~\cite[Theorem~4]{Bru03}.
The proof of the Main Theorem is based on the tools developed by Bruin in~\cite{Bru03}, and by Bruin, Keller and St.~Pierre in~\cite{BruKelsPi97}.

One of the interesting features of the Main Theorem, in constrast with the other realization results mentioned above, is that the systems we consider have a natural differentiable structure.
It turns out that, for the parameters~$\parameter \in (0, 4]$ given by (the proof of) the Main Theorem, the invariant measures supported by the \pcs{} of~$f_\parameter$ correspond precisely to those invariant measures~$\mu$ of~$f_\parameter$ whose Lyapunov exponent
$$ \chi(\mu) \= \int \ln |f_\parameter'| d\mu, $$
vanishes~\cite[Lemma~21]{CorRiv0804}.
It also turns out that every invariant probability measure supported on the \pcs{} of~$f_\parameter$ is an ``equilibrium state of~$f_\parameter$ for the potential~$- \ln |f_\lambda'|$''.
That is, if for each invariant measure~$\mu$ we denote by~$h_\mu$ its measure theoretic entropy, then the supremum
$$ \sup \left\{ h_\mu - \chi(\mu) \mid \mu \text{ invariant probability measure of~$f_\parameter$} \right\}, $$
is attained at each invariant probability measure supported by the \pcs{} of~$f_\parameter$, see~\cite[Lemma~21]{CorRiv0804}.

We thus obtain the following corollary of the Main Theorem.
\begin{coro}\label{c:ergodic theory interval}
For each non-empty metrizable Choquet simplex~$\Choquet$ there is a parameter~$\parameter \in (0, 4]$ verifying the conclusions of the Main Theorem, and such that in addition the space of invariant probability measures of~$f_\parameter$ (resp. equilibrium states of~$f_\parameter$ for the potential $- \log |f_\parameter'|$) that are of zero Lyapunov exponent, endowed with the weak$^*$ topology, is affine homeomorphic to~$\sC$.
\end{coro}
This result is in sharp contrast with the fact that for a logistic map there can be at most one ergodic equilibrium state whose Lyapunov exponent is strictly positive.\footnote{This follows from a result of Leddrapier in~\cite{Led81}, that every such equilibrium state is absolutely continuous with respect to the Lebesgue measure, and from the fact that logistic maps are ergodic with respect to the Lebesgue measure~\cite{BloLyu91}.}

For future reference we state an holomorphic version of Corollary~\ref{c:ergodic theory interval}, shown in~Appendix~\ref{a:ergodic theory}.
For a complex parameter~$\parameter \in \C$ denote by~$P_\parameter$ the quadratic polynomial defined by
$$ P_\parameter (z) = \parameter z (1 - z), $$
viewed as a dynamical system acting on~$\C$.
\begin{coro}\label{c:ergodic theory holomorphic}
For each non-empty metrizable Choquet simplex~$\Choquet$ there is a parameter~$\parameter \in (0, 4]$ verifying the conclusions of the Main Theorem, and such that in addition, if we denote by~$t_0$ the Hausdorff dimension of the Julia set of~$P_\parameter$, then the space of invariant probability measures of~$P_\parameter$ (resp. equilibrium states of~$P_\parameter$ for the potential $- t_0 \log |P_\parameter'|$) that are of zero Lyapunov exponent, endowed with the weak$^*$ topology, is affine homeomorphic to~$\sC$.
\end{coro}

We end this introduction by stating some questions that arise naturally from the Main Theorem.
To do this, for each~$\parameter \in (0, 4]$ we will denote by~$X_\parameter$ the \pcs{} of~$f_\parameter$.
Given a compact metrizable topological space~$X$ and a continuous map~$T : X \to X$, the Main Theorem implies that there is a parameter~$\parameter \in (0, 4]$ and an affine homeomorphism~$H$ between the space of invariant probability measures supported on~$X_\parameter$.
It is thus natural to ask whether the parameter~$\parameter \in (0, 4]$ and~$H$ can be chosen in such a way that~$H$ is induced by a continuous map between~$X_\parameter$ and~$X$.
More precisely, the question is if~$\parameter \in (0, 4]$ and~$H$ can be chosen in such a way that there is a continuous map~$h : X_\parameter \to X$ such that for each invariant probability measure~$\mu$ supported by~$X_\parameter$ we have~$H(\mu) = h_* \mu$.

This type of problem is very well understood in the setting of minimal Cantor systems: Giordano, Putnam, and Skau have shown in~\cite{GioPutSka95} that for two minimal Cantor systems~$(X, T)$ and $(X',T')$ there exists a homeomorphism $h : X\to X'$ that induces an affine homeomorphism between the corresponding spaces of invariant probability measures, if, and only if,~$(X, T)$ and~$(X', T')$ are ``orbit equivalent'': There is a homeomorphism between~$X$ and~$X'$ mapping each orbit of~$T$ to an orbit of~$T'$.
Furthermore, to each minimal Cantor system~$(X, T)$ one can associate a dimension group that is a complete invariant for the orbit equivalence relation~\cite{GioPutSka95}: Two minimal Cantor systems are orbit equivalent if, and only if, the corresponding dimension groups are isomorphic as ordered groups with unit.\footnote{The dimension group is defined as the quotient of the space of continuous functions defined on~$X$ and taking values in~$\Z$, by the subgroup of functions whose integral with respect to each invariant measure vanishes; the positive cone is the set of those classes containing a function taking values in~$\N_0$, and the unit is the class of the constant function equal to~1. See for example~\cite[Theorem~1.13]{GioPutSka95}.}

It is thus natural to look for a special class of minimal Cantor systems realizing all of the orbit equivalent classes.
Since the dimension group associated to each Toeplitz flow contains the dimension group of an odometer as a subgroup \cite[Section 4.1]{GjeJoh00}, the class of Toeplitz flows is not sufficient to realize all orbit equivalence classes, in spite of the fact that this class realizes all the non-empty metrizable Choquet simplices as sets of invariant probability measures~\cite{Dow91}.

So the question remains whether minimal \pcs s of logistic maps realize all orbit equivalence classes.
In order to formulate a precise question we will consider natural extensions to stay in the class of minimal Cantor systems, and use the generalized odometer associated to a kneading map, see~\S\S\ref{ss:unimodal}, \ref{ss:kneading odometer} for definitions.
\begin{ques}
Does every orbit equivalence class contain the natural extension of a generalized odometer associated to a kneading map?
\end{ques}
\begin{ques}
Does every uniquely ergodic orbit equivalence class contain the natural extension of a generalized odometer associated to a kneading map?
\end{ques}
It is well-known that every odometer can be realized, up to a homeomorphism, as the \pcs{} of an infinitely renormalizable logistic map, see also~\cite{BloKeeMis06}.
In~\S\ref{ss:example} we give an example of a uniquely ergodic generalized odometer associated to a kneading map, whose natural extension is not orbit equivalent to an odometer, nor to a Toeplitz flow.

\subsection{Notes and references}
\label{ss:notes and references}
Although Corollary~\ref{c:extreme realization} is stronger than~\cite[Main Theorem]{CorRiv0804}, we use this last result in the proof of the Main Theorem to deal with case of finite dimensional Choquet simplices.

We have stated the Main Theorem and Corollary~\ref{c:extreme realization} for the logistic family for simplicity.
We show that an analogous statement holds for each full family of unimodal maps, as well as for the family of symmetric tent maps.
See~\S\ref{ss:unimodal} for definitions.
In fact, for each infinitely dimensional metrizable Choquet simplex we construct kneading map~$Q$ such that the conclusions of the Main Theorem hold for each unimodal map whose kneading map is~$Q$, see~\S\ref{s:doubly resonant kneading}.
Furthermore the kneading map satisfies for every~$k \in \N_0$ the inequality~$Q(k) \le \max \{ 0, k - 2 \}$ (part~1 of Lemma~\ref{l:doubly resonant kneading}), and therefore every full family of unimodal maps, as well as the family of symmetric tent maps, contains a unimodal map whose kneading map is~$Q$.
For the case of finite dimensional Choquet simplices see~\cite[Remark~1]{CorRiv0804}.

Similarly, it follows from~\cite[Lemma~21]{CorRiv0804} that Corollary~\ref{c:ergodic theory interval} holds for every full family of \Sunimodal{} maps.

See~\cite[\S15]{Dow05} for a survey on realization results concerning Toeplitz flows.

See~\cite{GamMar06} for the realization of some concrete simplices as the space of invariant measures of minimal Cantor systems.

\subsection{Strategy and organization}
In this section we explain the strategy of the proof of the Main Theorem and simultaneously describe the organization of the paper.

We only deal with infinitely dimensional Choquet simplices, the finite dimensional case being covered by~\cite[Main Theorem]{CorRiv0804}.
We use a result of Lazar and Lindenstrauss that characterizes infinite dimensional metrizable Choquet simplices as inverse limits of stochastic matrices, see Theorem~\ref{t:Choquet as inverse limit} in~\S\ref{ss:Choquet simplices}.

We describe the logistic maps in the Main Theorem through their associated ``kneading map'', see~\S\ref{ss:unimodal} for the definition of kneading map and further background on unimodal maps.
In fact, the conclusions of the Main Theorem are valid for each unimodal map having the same kneading map as~$f_\parameter$.
To ensure that the \pcs{} is a Cantor set where the unimodal map is minimal, it is enough to require that the kneading map diverges to~$+\infty$ (Proposition~\ref{p:continuous unimodal}).

In~\S\ref{ss:doubly resonant kneading} we introduce a class of kneading maps that diverge to~$+ \infty$.
In~\S\ref{ss:proof of Main Theorem} we state a result describing the space of invariant measures supported on the \pcs{} of a unimodal map with a kneading map in this class (Theorem~\ref{t:invariant of doubly resonant}).
In~\S\ref{ss:proof of Main Theorem} we also give a proof of the Main Theorem assuming Theorem~\ref{t:invariant of doubly resonant}.

In~\S\ref{ss:kneading odometer} we recall the definition and some properties of the generalized odometer associated to a kneading map, that was introduced in~\cite{BruKelsPi97}.
In~\S\ref{ss:measure isomorphism} we show that for a unimodal map whose kneading map~$Q$ is as in~\S\ref{ss:doubly resonant kneading}, the space of invariant probability measures supported by the \pcs{} is affine homeomorphic to that of the generalized odometer associated to~$Q$.
In turn, this space is affine homeomorphic to the corresponding space of the Bratteli-Vershik system associated to~$Q$, introduced in~\cite{Bru03}; see~\cite[Proposition~2]{Bru03} or Theorem~\ref{t:reduction to Bratteli} in~\S\ref{ss:kneading Bratteli-Vershik}.
The advantage of this last space is that it can be described explicitly as an inverse limit of some ``transition matrices'', see~\S\ref{ss:Bratteli-Vershik}. 
We calculate the transition matrices and some of their products in~\S\ref{ss:transition matrices}, and give the proof of Theorem~\ref{t:invariant of doubly resonant} in~\S\ref{ss:proof of invariant of doubly resonant}.

In~\S\ref{ss:example} we give an example of a uniquely ergodic generalized odometer associated to a kneading map whose natural extension is not orbit equivalent to a Toeplitz flow.

In Appendix~\ref{a:ergodic theory} we give the proof of Corollary~\ref{c:ergodic theory holomorphic}.

\subsection{Acknowledgments}
We are grateful with Christian Skau for his help with dimension groups of Toeplitz flows, and an anonymous referee for his/her suggestions to improve the presentation of the paper.

\section{Preliminaries}\label{s:preliminaries}
After fixing some notation in~\S\ref{ss:linear algebra}, we review some concepts and results about Choquet simplices (\S\ref{ss:Choquet simplices}), and unimodal maps~(\S\ref{ss:unimodal}).

Throughout the rest of this article~$\N$ denotes the ring of strictly positive integers and $\N_0 \= \N \cup \{ 0 \}$.
We will use the interval notation for subsets of~$\N_0$: For $n, n' \in \N_0$ such that $n' \ge n$, we put
$$ [n, n'] \= \{ k \in \N_0 \mid n \le k \le n' \}, $$
and we put~$[n, n'] \= \emptyset$ when $n' < n$.

\subsection{Linear algebra}\label{ss:linear algebra}
Given a non-empty finite set~$V$, for each $v \in V$ we denote by $\vec{e}_v \in \R^V$ the vector having all of its coordinates equal to~$0$, except for the coordinate corresponding to~$v$ that is equal to~$1$.
Notice in particular that $\{ \vec{e}_v \mid v \in V \}$ is a base of~$\R^V$.
Furthermore we will denote by~$\Delta_{V}$ the \textit{unit simplex} in $\R^V$, which is defined as the (closed) convex hull of $\{ \vec{e}_v \mid v \in V \}$ in~$\R^V$, and by~$\| \cdot \|_1$ the norm on~$\R^V$ defined by $\| \sum_{v \in V}\alpha_v \vec{e}_v \|_1 = \sum_{v \in V} |\alpha_v|$.
Observe that~$\| \cdot \|_1$ is constant equal to~1 on~$\Delta_V$. 

Given non-empty finite sets $V, V'$ denote by $\cM_{V, V'}$ the group of matrices whose entries are real and indexed by $V \times V'$.
For a matrix $A \in \cM_{V, V'}$ we denote by $A^t$ the transpose of~$A$, and for $(v, v') \in V \times V'$ we denote by $A(v, v')$ the corresponding entry of~$A$, and by $A(\cdot, v')$ the corresponding column vector of~$A$.
Given column vectors $\{ \vec{x}_{v'} \mid v' \in V' \}$ in $\R^V$ we denote by $(\vec{x}_{v'})_{v' \in V'}$ the matrix in $\cM_{V, V'}$ whose column vector corresponding to the coordinate~$v'$ is equal to $\vec{x}_{v'}$.

We say that a matrix~$A$ is (left) \textit{stochastic} if all of its entries are non-negative and if the sum of all the entries in each column is equal to~1.
Observe that a stochastic matrix in $\cM_{V, V'}$ maps~$\Delta_{V'}$ into~$\Delta_V$, and that the product of stochastic matrices is stochastic.
\begin{lemm}\label{l:non expanding}
Let $V, V'$ be non-empty finite sets and let $A \in \cM_{V, V'}$ be a stochastic matrix.
Then for each $\vec{w}, \vec{w}' \in \Delta_V$ we have
$$ \| A(\vec{w}) - A(\vec{w}') \|_1 \le \| \vec{w} - \vec{w}' \|_1. $$
\end{lemm}
\begin{proof}
Putting $\vec{w} = ( w_v )_{v \in V}$ and $\vec{w}' = ( w_{v}' )_{v \in V'}$, we have
\begin{multline*}
\| A(\vec{w}) - A(\vec{w}') \|_1
\le
\sum_{v \in V} \| (w_v - w_v') A(\vec{e}_v) \|_1
=
\sum_{v \in V} |w_v - w_v'|
=
\| \vec{w} - \vec{w}' \|_1.
\end{multline*}
\end{proof}
\subsection{Choquet simplices}\label{ss:Choquet simplices}
A compact, convex, and metrizable subset~$\Choquet$ of a locally convex real vector space is said to be a (metrizable) \textit{Choquet simplex}, if for each $v \in \Choquet$ there is a unique probability measure~$\mu$ that is supported on the set of extreme points of~$\Choquet$, and such that $\int x d\mu(x) = v$.
See for example~\cite[{\S}II.3]{Alf71} for several characterizations of Choquet simplices.

In the proof of the Main Theorem we will make use of the following characterization of infinite dimensional metrizable Choquet simplices.
\begin{theo}[\cite{LazLin71}, Corollary, p.~186]
\label{t:Choquet as inverse limit}
Given an infinite dimensional Choquet simplex~$\Choquet$, for each~$n \in \N$ there is a surjective affine map $A_n : \Delta_{[0, n + 1]} \to \Delta_{[0, n]}$ such that $\varprojlim_{n} (\Delta_{[0, n + 1]}, A_n)$ is affine homeomorphic to~$\Choquet$.
\end{theo}
The following lemma is a simple consequence of the previous theorem.
We will say that a sequence~$(A_n)_{n \in \N}$ as in the theorem is \textit{normalized}, if for every $n \in \N$ and $j \in [0, n]$ we have $A_n(\vec{e}_j) = \vec{e}_{j}$.
\begin{lemm}
  \label{l:normalized Choquet}
For every infinite dimensional metrizable Choquet simplex~$\Choquet$ there is a normalized sequence of linear maps~$(A_n)_{n \in \N}$ such that~$\varprojlim_{n}(\Delta_{[0, n + 1]}, A_n)$ is affine homeomorphic to~$\Choquet$.
\end{lemm}
\begin{proof}
Let~$(\tA_n)_{n \in \N}$ be a sequence of affine maps given by Theorem~\ref{t:Choquet as inverse limit}.

For each~$n \in \N$ define a permutation~$\sigma_n$ of~$[0, n]$ by induction as follows.
Let~$\sigma_1$ be the identity, and suppose that for some~$n \in \N$ the permutation~$\sigma_n$ is already defined.
Since~$\tA_n$ maps~$\Delta_{[0, n + 1]}$ surjectively onto~$\Delta_{[0, n]}$, for each~$j \in [0, n]$ there is~$k \in [0, n + 1]$ such that~$\tA_n(\vec{e}_{k}) = \vec{e}_j$.
Equivalently, there is a map~$\iota_n : [0, n] \to [0, n + 1]$ such that for each $j \in [0, n]$ we have $\tA_n(\vec{e}_{\iota_n(j)}) = \vec{e}_j$.
Let~$\sigma_{n + 1}$ be the unique permutation of~$[0, n + 1]$ such that for each $j \in [0, n]$ we have $\sigma_{n + 1}(\iota_n(j)) = \sigma_n (j)$.

For each~$n$ let~$H_n : \R^{[0, n]} \to \R^{[0, n]}$ be the linear map so that for each~$j \in [0, n]$ we have~$H_n(\vec{e}_j) = \vec{e}_{\sigma_n(j)}$.
Then, by the definition of $(\sigma_n)_{n \in \N}$ it follows that for each $n \in \N$ the linear map $A_n \= H_n \circ \tA_n \circ H_{n + 1}^{-1}$ maps~$\Delta_{[0, n + 1]}$ surjectively onto~$\Delta_{[0, n]}$, and that for every~$j \in [0, n]$ we have $A_n(\vec{e}_{j}) = \vec{e}_j$.
Therefore the sequence of linear maps~$(A_n)_{n \in \N}$ is normalized, and $(H_n)_{n \in \N}$ induces a linear homeomorphism between $\varprojlim_{n} (\Delta_{[0, n + 1]}, \tA_n)$ and $\varprojlim_{n} (\Delta_{[0, n + 1]}, A_n)$.
\end{proof}
We end this section with the following general lemma.
\begin{lemm}\label{l:equivalence}
For each $n \in \N$ let $A_n, B_n : \R^{[0, n + 1]} \to \R^{[0, n]}$ be stochastic matrices such that
$$ \sum_{n \in \N} \sup \{ \| A_n(\vec{v}) - B_n(\vec{v}) \|_1 \mid  \vec{v} \in \Delta_{[0, n + 1]} \} < + \infty.
$$
Then the inverse limits $\varprojlim_n (\Delta_{[0, n + 1]}, A_n)$ and $\varprojlim_n (\Delta_{[0, n + 1]}, B_n)$ are affine homeomorphic.

In particular, if both~$(A_n)_{n \in \N}$ and~$(B_n)_{n \in \N}$ are normalized and
$$ \sum_{n \in \N} \| A_n(\vec{e}_{n + 1}) - B_n(\vec{e}_{n + 1}) \|_1 < + \infty, $$
Then the inverse limits $\varprojlim_n (\Delta_{[0, n + 1]}, A_n)$ and $\varprojlim_n (\Delta_{[0, n + 1]}, B_n)$ are affine homeomorphic.
\end{lemm}
\begin{proof}
Let $x \= (\vec{x}_n)_{n\in \N}\in \varprojlim_{n} (\Delta_{[0,n+1]},A_n).$
For each $n, m \in\N$ such that $m\geq n$ define
$$
\vec{x}_{n,m} \= B_n\cdots B_m \vec{x}_{m+1} \in \Delta_{[0,n]}.
$$
In particular we have $\vec{x}_{n, n} = \vec{x}_n$.

When $m > 0$ we have
\begin{eqnarray*}
\|\vec{x}_n-\vec{x}_{n,m}\|_1 & = & \|A_n\cdots A_m\vec{x}_{m+1}- B_n\cdots B_m
\vec{x}_{m+1}\|_1\\
  & \leq & \|A_n\cdots A_m\vec{x}_{m+1}-B_nA_{n+1}\cdots A_m\vec{x}_{m+1}\|_1\\
  & & +
  \|B_nA_{n+1}\cdots A_m\vec{x}_{m+1}-B_n\cdots B_m
\vec{x}_{m+1}\|_1.
\end{eqnarray*}
From Lemma~\ref{l:non expanding} we get
$$
 \|B_nA_{n+1}\cdots A_m\vec{x}_{m+1}-B_n\cdots B_m\vec{x}_{m+1}\|_1\leq \|A_{n+1}\cdots
 A_m\vec{x}_{m+1}-B_{n+1}\cdots B_m\vec{x}_{m+1}\|_1,
$$
and  since
$$
\|A_n\cdots A_m\vec{x}_{m+1}-B_nA_{n+1}\cdots A_m\vec{x}_{m+1}\|_1\leq
\sup\{v\in \Delta_{[0,n+1]}: \|A_nv-B_nv\|_1\},
$$
we deduce, after  an induction argument, that
\begin{eqnarray}
\label{eq3} \|\vec{x}_n-\vec{x}_{n,m}\|_1
\leq
\sum_{k=n}^m\sup\{\|A_kv-B_kv\|_1 \mid v\in \Delta_{[0,k+1]} \}.
\end{eqnarray}
By hypothesis and from equation (\ref{eq3}), we deduce that for a fixed~$n \in \N$, sequence $(\vec{x}_{n,m})_{m \geq n}$ is a Cauchy sequence in~$\Delta_{[0,n]}$.
We denote by~$H_n(x)$ its limit.
Observe that for each~$n \in \N$, the sequence $(B_n\vec{x}_{n+1,m})_{m \geq n+1}$ converges to both $B_nH_{n + 1}(x)$ and $H_n(x)$.
This implies that $(H_n(x))_{n\in \N}\in \varprojlim_{n}(\Delta_{[0,n+1]},B_n).$
Thus the transformation
$$H : \varprojlim_{n}(\Delta_{[0,n+1]},A_n) \to \varprojlim_{n}(\Delta_{[0,n+1]},B_n)$$ given by
$H(x)=(H_n(x))_{n \in \N}$ is well defined.
This map is clearly affine.
We will show that~$H$ is a homeomorphism.

In order to verify that~$H$ is continuous, we just need to show that for each~$m \in \N$ the map~$H_m$ is continuous.
As~$\varprojlim_{n} (\Delta_{[0, n + 1]}, A_n)$ is metrizable we just need to show that~$H_m$ is sequentially continuous.
To do this, fix~$\varepsilon > 0$ and consider a sequence
$x^{(n)} \= ((\vec{x}_m^{(n)})_{m\in\N})_{n\in\N}$ in $\varprojlim_{n}(\Delta_{[0,n+1]},A_n)$ that converges to $x = (\vec{x}_n)_{n\in\N}$.
For each $n, m, k \in \N$ such that $k \ge m$ we have
\begin{eqnarray*}
\| H_m(x^{(n)}) - H_m(x) \|_1
& \leq &
\|H_m(x^{(n)}) - B_m\cdots B_k \vec{x}^{(n)}_{k+1}\|_1
\\ && +
\|B_m\cdots B_k\vec{x}^{(n)}_{k+1} - B_m\cdots B_k \vec{x}_{k+1}\|_1
\\ && +
\|B_m \cdots B_k \vec{x}_{k+1} - H_m(x)\|_1
\\ &\leq &
\|H_m(x^{(n)})-B_m\cdots B_k \vec{x}^{(n)}_{k+1}\|_1
\\ && +
\|\vec{x}^{(n)}_{k+1}-\vec{x}_{k+1}\|_1
\\ && +
\|B_m\cdots B_k\vec{x}_{k+1}-H(\vec{x}_m)\|_1.
\end{eqnarray*}
We choose~$k$ sufficiently large so that
$$
\| H_m(x^{(n)}) - B_m\cdots B_k \vec{x}^{(n)}_{k+1}\|_1,
\| B_m\cdots B_k \vec{x}_{k+1} - H_m(x)\|_1
\leq \frac{\varepsilon}{3},
 $$
and~$n$ such that  $\|x^{(n)}_{k+1}-\vec{x}_{k+1}\|_1\leq \varepsilon/3$.
We get
$$
\|H_m(x^{(n)})-H_m(x)\|_1 \leq \varepsilon.
$$
This shows the continuity of~$H_m$, and hence that of~$H$.

To show that~$H$ is a homeomorphism we define in a similar way,
$$
L : \varprojlim_{n}(\Delta_{[0,n+1]},B_n) \to \varprojlim_{n}(\Delta_{[0,n+1]},A_n).$$
Observe that by Lemma~\ref{l:non expanding} we get
\begin{eqnarray*}
\|L_n(H_n(x)) - x \|_1
& \leq &
\|L_n(H_n(x)) - A_n\cdots A_mH_{m + 1}(x) \|_1
\\ && +
\|A_n \cdots A_m H_{m + 1}(x) - A_n \cdots A_m \vec{x}_{m+1} \|_1
\\ & \leq &
\| L_n(H_n(x)) - A_n\cdots A_m H_{m + 1}(x) \|_1
\\ && +
\| H_{m + 1}(x) - \vec{x}_{m+1} \|_1
\end{eqnarray*}
By definition of~$L$, we have $\lim_{m\to \infty}\|L_n(H_n(x)) - A_n\cdots A_mH_{m + 1}(x) \|_1 = 0,$  and from~\eqref{eq3} we have $\lim_{m\to
\infty}\| H_{m + 1}(x) - \vec{x}_{m+1}\|_1=0$.
Thus we conclude that~$L$ and~$H$ are inverse of each other.
\end{proof}
\subsection{Unimodal maps, cutting times and the kneading map}\label{ss:unimodal}
A continuous map $f:[0,1]\to [0,1]$ is \textit{unimodal} if~$f(0) = f(1) = 0$, and if there exists a point $c \in [0,1]$ such that~$f$ is strictly increasing on~$[0, c]$, and strictly decreasing on~$[c, 1]$.
The point~$c$ is called the \textit{turning} or \textit{critical point} of~$f$.
For each $\parameter \in (0, 4]$ the logistic map~$f_\parameter$ is a unimodal with critical point~$x = \tfrac{1}{2}$.

Let~$f$ be a unimodal map with critical point~$c$.
The $\omega$\nobreakdash-limit of~$c$ will be called the \textit{\pcs{}} of~$f$.
When either $f(c) \le c$ or $f^2(c) \ge c$, it is easy to see that the \pcs{} of~$f$ reduces to a single point.
We will thus (implicitly) assume from now on that for each unimodal map~$f$ that we consider we have $f^2(c) < c < f(c)$.

To describe the dynamics of a unimodal map~$f$ on its \pcs, we will make the following definitions.
Let~$c$ be the critical point of~$f$ and for each $n \ge 1$ put $c_n = f^n(c)$.
Define the sequence of compact intervals $(D_n)_{n \ge 1}$ inductively by $D_1 = [c, c_1]$, and for each $n \ge 2$, by
$$
D_{n} =
\begin{cases}
f(D_{n - 1}) & \text{if } c \not \in D_{n - 1}, \\
[c_{n}, c_1] & \text{otherwise}.
\end{cases}
$$
An integer $n \ge 1$ will be called a \textit{cutting time} if $c \in D_n$.
We will denote by $(S_k)_{k \ge 0}$ the sequence of all cutting times.
From our assumption that $f^2(c) < c < f(c)$ it follows that $S_0 = 1$ and $S_1 = 2$.

It can be shown that if~$S$ and~$S'> S$ are consecutive cutting times, then $S' - S$ is again a cutting time, and that this cutting time is less than or equal to~$S$ when~$f$ has no periodic attractors, see for example~\cite{Bru95,Hof80}.
That is, if~$f$ has no periodic attractors then for each $k \ge 1$ there is a non-negative integer~$Q(k)$, such that $Q(k) \le k - 1$, and
$$ S_k - S_{k - 1} = S_{Q(k)}. $$
Putting $Q(0) = 0$, the function $Q : \N_0 \to \N_0$ so defined is called the \textit{kneading map} of~$f$.
It follows from the recursion formula above, and from $S_0 = 1$, that the sequence $(S_k)_{k \ge 0}$ of cutting times is determined by~$Q$.

We will say that a function $Q : \N_0 \to \N_0$ is a \textit{kneading map} if there is a unimodal map~$f$ with critical point~$c$, such that $f^2(c) < c < f(c)$, such that~$f$ has no periodic attractors and such that the kneading map of~$f$ is equal to~$Q$.
If we denote by~~$\succeq$ the lexicographical ordering in~$\N_0^{\N_0}$, then a function $Q : \N_0 \to \N_0$ is a kneading map if and only if $Q(0) = 0$, for each $k \ge 1$ we have $Q(k) \le k - 1$, and if for each $k \ge 1$ we have
\begin{equation}
\label{admisible}
\{Q(k+j)\}_{j\geq 1}
\succeq
\{Q(Q (Q(k)) + j)\}_{j\geq 1},
\end{equation}
see~\cite{Bru95,Hof80}.
Notice in particular that, if $Q : \N_0 \to \N_0$ is non-decreasing, $Q(0) = 0$ and for each $k \ge 1$ we have $Q(k) \le k - 1$, then~$Q$ is a kneading map.

We will need the following well known facts, see for example the proof of~\cite[Proposition~4]{CorRiv0804} for precise references.
\begin{prop}\label{p:continuous unimodal}
Let~$f$ be a unimodal map whose kneading map diverges to~$+\infty$.
Then the \pcs{} of~$f$ is a Cantor set, and the restriction of~$f$ to this set is minimal and has zero topological entropy.
Furthermore, if~$\widehat{f}$ is a unimodal map having the same kneading map as~$f$, then the space of invariant probability measures of~$\widehat{f}$ supported on the \pcs{} of~$\widehat{f}$ is affine homeomorphic to that of~$f$.
\end{prop}

\section{Reduced statement}\label{s:doubly resonant kneading}
The purpose of this section is to prove the Main Theorem assuming a result we state as Theorem~\ref{t:invariant of doubly resonant}.
We start introducing a class of kneading maps in~\S\ref{ss:doubly resonant kneading}.
In~\S\ref{ss:proof of Main Theorem} we state a result describing the space of invariant probability measures supported on the \pcs{} of a unimodal map with a kneading map in this class (Theorem~\ref{t:invariant of doubly resonant}).
In~\S\ref{ss:proof of Main Theorem} we also give a proof of the Main Theorem assuming Theorem~\ref{t:invariant of doubly resonant}.
\subsection{Kneading maps}\label{ss:doubly resonant kneading}
Throughout the rest of this paper we denote by $(r_n)_{n \in \N_0}$ the sequence of integers defined by $r_n = \frac{(n + 1)(n + 2)}{2}$.
Note that for each~$n \in \N$ we have $r_n = r_{n - 1} + n + 1$.

For each $n \in \N$ let $\vec{a}_n \= (a_{n, 0}, \ldots, a_{n, n}) \in \N^{[0, n]}$ be given and let $q \= (q_r)_{r \ge 0}$ be an increasing sequence of integers such that $q_0 = 0$, and such that for each $n \in \N$ we have,
\begin{equation}\label{e:jump}
q_{r_n} - q_{r_n - 1} = a_{n, 0} + \cdots + a_{n, n}.
\end{equation}
If we put $a \= (\vec{a}_n)_{n \in \N}$, then we will define a kneading map~$\kneading : \N_0 \to \N_0$ as follows.
For each $n \in \N_0$ put
$$ \lefti_n = [q_{r_n - 1} + 1, q_{r_n}]
\text{ and }
\righti_n = [q_{r_{n}} + 1, q_{r_{n + 1} - 1}],
$$
and note that $\{ \lefti_n, \righti_n \mid n \in \N \}$ is a partition of~$\N$.
Furthermore, for each $n \in \N$ and $m \in [0, n]$ define
$$ \lefti_{n, m} =
\left[ \left( q_{r_n - 1} + 1 + \sum_{i = 0}^{m - 1} a_{n, i} \right), \left( q_{r_n - 1} + \sum_{i = 0}^{m} a_{n, i} \right)\right]
$$
and
$$
\righti_{n, m} =
[q_{r_n + m} + 1, q_{r_n + m + 1}].
$$
By~\eqref{e:jump} the collection $\{ \lefti_{n, m} \mid m \in [0, n] \}$ is a partition of~$\lefti_{n}$.
Since for each $n \in \N$ we have $r_{n + 1} = r_n + n + 2$, the collection $\{ \righti_{n, m} \mid m \in [0, n] \}$ is a partition of~$\righti_{n}$.

With these notations we put
$$ \kneading
\=
\sum_{n \in \N} \sum_{m = 0}^n q_{r_{n - 1} + m} \left( \pmb{1}_{\lefti_{n, m}} + \pmb{1}_{\righti_{n, m}} \right). $$
Note that $\kneading^{-1}(0) = [0, q_2]$, $\kneading(\N_0) = \{ q_r \mid r \in \N_0 \}$ and that for each $n \in \N$ we have
\begin{multline*}
\kneading(\lefti_n)
=
\kneading(\righti_n)
= \{ q_{r_{n - 1}}, q_{r_{n -1} + 1}, \ldots, q_{r_{n - 1} + n} \}
\\ \subset
\lefti_{n - 1} \cup \righti_{n - 1}.
\end{multline*}
\begin{lemm}\label{l:doubly resonant kneading}
For $a \= (\vec{a}_n)_{n \in \N}$, $q = (q_r)_{r \in \N_0}$ and~$\kneading$ as above, the following properties hold.
\begin{enumerate}
\item[1.]
For each $k \in \N_0$ we have $\kneading(k) \le \max \{ 0, k - 2 \}$.
\item[2.]
The function~$\kneading$ is a kneading map and for every $k \ge q_5 + 1$ we have
\begin{equation}
  \label{e:almost admissibility}
  \kneading(k + 1) \ge \kneading( \kneading(\kneading(k)) + 1) + 2.
\end{equation}
\end{enumerate}
\end{lemm}
\begin{proof}
Put~$Q \= \kneading$.

\partn{1}
For~$k \in [0, q_2]$ we have $Q(k) = 0$ so the inequality is satisfied in this case.
Let~$k \in \N$ be such that $k \ge q_2 + 1$, so there is~$n \in \N$ such that~$k \in \lefti_n \cup \righti_n = [q_{r_n - 1} + 1, q_{r_{n + 1} - 1}]$.
Thus~$Q(k) \le q_{r_{n - 1} + n} = q_{r_n - 1}$, so when~$k \neq q_{r_n - 1} + 1$ we have~$Q(k) \le k - 2$.
Finally observe that
$$ Q(q_{r_n - 1} + 1) = q_{r_{n - 1}} \le q_{r_{n - 1} + n} - n \le q_{r_n - 1} -1. $$

\partn{2}
Let~$k \in \N$ be such that~$k \ge q_5$, so there is $n \ge 2$ such that~$k + 1 \in \lefti_n \cup \righti_n$.
Then~$Q(k + 1) \ge q_{r_{n - 1}}$, $Q(k) \le q_{r_{n - 1} + n}$ and therefore we have $Q(Q(k)) \le q_{r_{n - 2} + n - 1}$ and
\begin{multline*}
Q(Q(Q(k)) + 1)
\le
q_{r_{n - 2}}
\le
q_{r_{n - 1}} - (r_{n - 1} - r_{n - 2})
\\ \le
q_{r_{n - 1}} - 2
\le
Q(k + 1) - 2.
\end{multline*}

In view of part~1 and the previous inequality, to show that~$Q$ is admissible we just need to show that for each ~$k \in [1, q_5 - 1]$ and $j \in [1, q_2 + 1 - k]$ we have
\begin{equation}
  \label{e:order}
  Q(k + j) \ge Q(Q(Q(k)) + j),
\end{equation}
with strict inequality when~$j = q_2 + 1 - k$.
In fact, for each $k \in [1, q_5 - 1]$ we have $Q(Q(k)) = 0$, so for each $j \in [1, q_2 + 1 - k]$, we have $Q(Q(Q(k)) + j) = 0$ and~\eqref{e:order} is satisfied.
When~$j = q_2 + 1 - k$ we have $Q(k + j) = Q(q_2 + 1) = q_1 > 0$, so inequality~\eqref{e:order} is strict in this case.
\end{proof}
\subsection{Reduced statement}
\label{ss:proof of Main Theorem}
The purpose of this section is to give a proof of the Main Theorem assuming the following one.
\begin{theoalph}\label{t:invariant of doubly resonant}
For each $n \in \N$ let $\vec{a}_n \in \N^{[0, n]}$ be given and put $a \= (\vec{a}_n)_{n \in \N}$.
Furthermore, let $q \= (q_r)_{r \in \N_0}$ be a strictly increasing sequence of integers such that~$q_0 = 0$ and such that for each $n \in \N$ we have
$$ q_{r_n} - q_{r_n - 1} = a_{n, 0} + \cdots + a_{n, n}, $$
and let $\kneading$ be the corresponding kneading map.
Define $(S_k)_{k \in \N_0}$ recursively by $S_0 = 1$ and $S_k = S_{k - 1} + S_{\kneading(k)}$, and assume that
\begin{equation}\label{e:positive determinant}
\sum_{r \in \N \setminus \{ r_n \mid n \in \N \}} \frac{S_{q_{r - 1}}}{S_{q_{r}}} < + \infty.
\end{equation}
Moreover, for each $n \in \N$ let $\Xi_n : \R^{[0, n + 1]} \to \R^{[0, n]}$ be the stochastic matrix such that for each $m \in [0, n]$ we have $\Xi_n(\vec{e}_m) = \vec{e}_{m}$ and such that
$$
\Xi_n(0, n + 1)
\=
\frac{S_{q_{r_{n} - 1}}}{S_{q_{r_{n}}}}(1 + a_{n, n})\vec{e}_0
+
\sum_{m = 1}^{n} \frac{S_{q_{r_{n - 1} + n - m}}}{S_{q_{r_{n}}}} a_{n, n - m} \vec{e}_m.
$$

Then for each unimodal map~$f$ whose kneading map is equal to~$\kneading$, the \pcs{} of~$f$ is a Cantor set, $f$ is minimal on this set, and the space of invariant probability measures of~$f$ supported on this set is affine homeomorphic to $\varprojlim_{n} (\Delta_{[0, n + 1]}, \Xi_n)$.
\end{theoalph}

To prove the Main Theorem, we first remark that the case where the metrizable Choquet simplex~$\Choquet$ is finite dimensional is given by~\cite[Main Theorem]{CorRiv0804}.
So from now on we assume that~$\Choquet$ is infinite dimensional.
Then by Lemma~\ref{l:normalized Choquet} there is a sequence of normalized affine maps $(A_n)_{n \in \N}$ such that $\varprojlim_{n} (\Delta_{[0, n + 1]}, A_n)$ is affine homeomorphic to~$\Choquet$.
In view of Lemma~\ref{l:equivalence}, we just need to find~$a$ and~$q$ as in the statement of Theorem~\ref{t:invariant of doubly resonant}, for which~\eqref{e:positive determinant} is satisfied and such that
\begin{equation*}
 \sum_{n \in \N} \| \Xi_n(0, n + 1) - A_n(0, n + 1) \|_1 < + \infty.
\end{equation*}
This is shown in the following lemma, thus completing the proof of the Main Theorem.
\begin{lemm}\label{l:approximation}
For each~$n \in \N$ let~$\vec{y}_n \in \Delta_{[0, n]}$ be given.
Then there are~$a$ and~$q$ as in the statement of Theorem~\ref{t:invariant of doubly resonant} for which~\eqref{e:positive determinant} is satisfied, and such that
\begin{equation}
  \label{e:summable approximation}
 \sum_{n \in \N} \| \Xi_n(0, n + 1) - \vec{y}_n \|_1 < + \infty.
\end{equation}
\end{lemm}
\begin{proof}
Given~$n \in \N$ and a non-zero vector $\vec{a} \in \R^{[0, n]}$ with non-negative coordinates, we will denote by~$\llbracket \vec{a} \rrbracket$ the unique vector in~$\Delta_{[0, n]}$ proportional to~$\vec{a}$.

We will define $a \= (\vec{a}_n)_{n \in \N}$ and~$q \= (q_r)_{r \in \N_0}$ by induction as follows.
Put~$q_0 = 0$, fix~$q_1 \ge 1$, and assume that for some~$n \in \N$ the numbers $q_2, \ldots, q_{r_{n - 1}}$ and the vectors~$\vec{a}_1, \ldots, \vec{a}_{n - 1}$ are already defined, in such a way that for each $m \in [0, n - 1]$ we have
$$ q_{r_m} - q_{r_m - 1} = a_{m, 0} + \cdots + a_{m, m}. $$
For each $r \in [r_{n - 1} + 1, r_{n - 1} + n]$ let~$q_r$ be defined in such a way that
$$ q_{r}
\ge
q_{r - 1} + r^2 \prod_{s=0}^{r - 2}(1 + q_{s + 1} + q_s),
$$
Note that these choices determine~$S_0, \ldots, S_{q_{r_{n - 1} + n}}$.

\partn{1}
We will show now that for each~$r \in [r_{n - 1} + 1, r_{n - 1} + n]$ we have
\begin{equation}\label{e:S_q growth}
\frac{S_{q_{r - 1}}}{S_{q_{r}}} \le r^{-2}.
\end{equation}
Using the recursion formula $S_{l}=S_{l-1}+S_{Q(l)}$ and $Q(l) \le l - 1$, we get by induction that for every $k, k' \in [1, q_{r_{n - 1} + n}]$ such that $k' < k$, we have $S_k \le S_{k'} (1 + k - k')$. 
In particular for every~$r \in [1, r_{n - 1} + n]$ we have
$$ S_{q_r} \le S_{q_{r - 1}} (1 + q_r - q_{r - 1}). $$
Since
$ S_{q_{r_0}} = S_1 = S_0(1 + q_1 - q_0) = 1 + q_1 - q_0 $, it follows by induction that for every~$r \in [1, r_{n - 1} + n]$ we have
$$
S_{q_r}\le \prod_{s=0}^{r - 1}(1 + q_{s + 1} - q_s).
$$
Hence for each~$r \in [r_{n - 1} + 1, r_{n - 1} + n]$ we have
\begin{multline*}
S_{q_{r}}
\ge
S_{q_{r - 1}} (1 + q_{r} - q_{r - 1})
\ge
q_{r} - q_{r - 1}
\\ \ge
r^2 \prod_{s=0}^{r - 2}(1 + q_{s + 1} + q_s)
\ge 
r^2 S_{q_{r - 1}},
\end{multline*}
as wanted.

\partn{2}
We will show that we can choose~$\vec{a}_n \= (a_{n, 0}, \ldots, a_{n, n}) \in \N^{[0, n]}$ is such a way that
\begin{equation}
  \label{e:integer approximation}
 \left\| \left\llbracket S_{q_{r_{n} - 1}}(1 + a_{n, n})\vec{e}_0
+
\sum_{m = 1}^{n} S_{q_{r_{n - 1} + n - m}} a_{n, n - m} \vec{e}_m \right\rrbracket - \vec{y}_n \right\|_1
\le n^{-2}.
\end{equation}
For~$x \in \R$ we denote by~$[x]$ integer part of~$x$.
Put~$\vec{y}_n = (y_0, \ldots, y_n)$, $N = \prod_{j = 0}^n S_{q_{r_{n - 1} + j}}$,
$k = (n + 2)^4$, for each~$j \in [0, n]$ put
$$ \zeta_j = N ([ky_j] + 1),
\text{ and put }
\zeta = \zeta_1 + \cdots + \zeta_n. $$
Since for each~$j \in [0, n]$ we have~$Nky_j \le \zeta_j \le Nky_j + N$, we obtain $kN \le \zeta \le (k + n + 1)N$.
So for each~$j \in [0, 1]$ we have
$$ \left| y_j - \frac{\zeta_j}{\zeta} \right|
\le
\left| y_j - \frac{kN y_j}{\zeta} \right| + \frac{N}{\zeta}
\le
\frac{(n + 1)N}{\zeta} + \frac{N}{\zeta}
\le
\frac{n + 2}{k}
=
\frac{1}{(n + 2)^3}.
$$
This shows that the vector~$\vec{\zeta} \= (\zeta_0, \ldots, \zeta_n) \in \N^{[0, n]}$ satisfies~$\| \llbracket \vec{\zeta} \rrbracket - \vec{y}_n \|_1 \le n^{-2}$.
Thus, if for each~$j \in [0, n - 1]$ we put $ a_{n, j} = \frac{\zeta_{n - j}}{S_{q_{r_{n - 1} + j}}}$, and if we put $ a_{n, n} \= \frac{\zeta_0}{S_{q_{r_{n - 1}}}} - 1 $, then~\eqref{e:integer approximation} is satisfies for this choice of~$\vec{a}_n$.
It remains to show that each of the coordinates of~$\vec{a}_n$ belongs to~$\N$.
By definition for each~$j \in [0, n]$ the integer~$\zeta_j$ is a strictly positive multiple of~$N$, so the coordinates of~$\vec{a}_n$ are integers and for each~$j \in [0, n - 1]$ we have $a_{n, j} \ge 1$.
Finally observe that by~\eqref{e:S_q growth} with $r = r_{n - 1} + 1$ we have,
$$ a_{n, n} \ge \frac{N}{S_{q_{r_{n - 1}}}} - 1 \ge S_{q_{r_{n - 1} + 1}} - 1 \ge (r_{n - 1} + 1)^2 - 1 \ge 1. $$

\partn{3}
Let~$\vec{a}_n$ be given by part~2 and put $q_{r_n} \= q_{r_{n - 1} + n} + a_{n, 0} + \cdots + a_{n, n}$.
This completes the inductive definition of~$q$ and~$a$.

To finish the proof of the lemma just observe that the inequalities~\eqref{e:S_q growth} imply~\eqref{e:positive determinant}, and the inequalities~\eqref{e:integer approximation} imply~\eqref{e:summable approximation}.
\end{proof}

\section{The generalized odometer and Bratteli-Vershik system associated to a kneading map}
\label{s:Bratteli-Vershik system}
The purpose of this section is to recall the definition of the generalized odometer and the Bratteli-Vershik system associated to a kneading map, that were introduced in~\cite{BruKelsPi97} and~\cite{Bru03}, respectively.
We start recalling the definition of the generalized odometer in~\S\ref{ss:kneading odometer}.
After briefly recalling the concepts of Bratteli diagram (\S\ref{ss:diagram}) and Bratteli-Vershik system (\S\ref{ss:Bratteli-Vershik}), we define the Bratteli-Vershik system associated to a kneading map in~\S\ref{ss:kneading Bratteli-Vershik}.
See for example~\cite{BarDowLia02,GraLiaTic95} for background on generalized odometers, and~\cite{DurHosSka99,HerPutSka92} and references therein for background and further properties of Bratteli-Vershik systems.

\subsection{The generalized odometer associated to a kneading map}\label{ss:kneading odometer}
Let $Q : \N_0 \to \N_0$ be a kneading map and put
\begin{multline*}
\Omega_Q \= \{ (x_k)_{k \ge 0} \in \{ 0, 1 \}^{\N_0} \mid x_k = 1 \text{ implies that for each}
\\
j = Q(k + 1), \ldots, k - 1 \text{ we have $x_j = 0$} \}.
\end{multline*}
If we denote by $(S_k)_{k \ge 0}$ the strictly increasing sequence of positive integers defined recursively by $S_0 = 1$ and $S_k = S_{k - 1} + S_{Q(k)}$,
it can be shown that for each non-negative integer~$n$ there is a unique sequence $\expansion{n} \= (x_k)_{k \ge 0}$ in~$\Omega_Q$, that has at most finitely many~$1$'s, and such that~$\sum_{k \ge 0} x_k S_k = n$.
The sequence $\expansion{n}$ is also characterized as the unique sequence in $\{0, 1 \}^{\N_0}$ with finitely many~$1$'s such that $\sum_{k \ge 0} x_k S_k = n$, and that it is minimal with this property with respect to the lexicographical order in~$\{0, 1 \}^{\N_0}$.

When~$Q$ diverges to~$+\infty$ the map defined on the subset~$ \{ \expansion{n} \mid n \in \N_0 \}$ of~$\Omega_Q$ by $\expansion{n} \mapsto \expansion{n + 1}$, extends continuously to a map~$T_Q : \Omega_Q \to \Omega_Q$ which is onto, minimal, and such that~$T_Q^{-1}$ is well defined on $\Omega_Q \setminus \expansion{0}$; see~\cite[Lemma~2]{BruKelsPi97}.
We call $(\Omega_Q, T_Q)$ the \textit{generalized odometer}\footnote{An odometer is a topological dynamical system defined in a similar way, in the case where the sequence of integers $(S_k)_{k \in \N_0}$ is such that for each~$k \in \N_0$ we have~$S_k | S_{k + 1}$.} associated to~$Q$.
Given $x = (x_k)_{k \ge 0} \in \Omega_Q$ and an integer $n \ge 0$, put $\sigma(x|n) = \sum_{k = 0}^nx_kS_k$.
Observe that $\sigma(x|n)$ is non-decreasing with~$n$, and when~$x$ has infinitely many~$1$'s, $\sigma(x|n) \to + \infty$ as $n \to + \infty$.
On the other hand, if~$x$ has at most a finite number of~$1$'s, then $\sigma(x) \= \lim_{n \to + \infty} \sigma(x|n)$ is finite and $x = \expansion{\sigma(x)}$.

For $x = (x_k)_{k \ge 0}$ different from $\expansion{0}$ we denote by~$q(x) \ge 0$ the least integer such that $x_{q(x)} \neq 0$. 
In~\cite[Theorem~1]{BruKelsPi97} it is shown that if $\parameter \in (0, 4]$ is a parameter such that the kneading map of the logistic map~$f_\parameter$ is equal to~$Q$, then for each $x \in \Omega_Q$ with infinitely many~$1$'s the sequence of intervals $(D_{\sigma(x|n)})_{n \ge q(x)}$ is nested and that $\bigcap_{n \ge q(x)} D_{\sigma(x|n)}$ is reduced to a point belonging to the \pcs~$X_{f_\parameter}$ of~$f_\parameter$.
Furthermore, if we denote this point by~$\pi(x)$ and for $n \ge 0$ we put $\pi(\expansion{n}) = f_\parameter^n(c)$, then the map $\pi : \Omega_Q \to X_{f_\parameter}$ so defined is continuous and conjugates the action of~$T_Q$ on~$\Omega_Q$, to the action of~$f_\parameter$ on~$X_{f_\parameter}$.

\subsection{Bratteli diagrams}\label{ss:diagram}
A \textit{Bratteli diagram} is an infinite directed graph $(V,E)$,
such that the vertex set~$V$ and the edge set~$E$ can be partitioned
into finite sets
$$
V = V_0\cup V_1 \cup  \cdots \mbox{ and } E=E_1\cup E_2\cup
\cdots
$$
with the following properties:
\begin{itemize}
\item
$V_0=\{v_0\}$ is a singleton.
\item
For every $j \ge 1$, each edge in~$E_j$ starts in a vertex in~$V_{j - 1}$ and arrives to a vertex in~$V_{j}$.
\item
All vertices in~$V$ have at least one edge starting from it, and all vertices except~$v_0$ have at least one edge arriving to it.
\end{itemize}
For a vertex $e \in E$ we will denote by~$s(e)$ the vertex where~$e$ starts and by~$r(e)$ the vertex to which~$e$ arrives.
A \textit{path} in $(V, E)$ is by definition a finite (resp. infinite) sequence $e_1e_2 \ldots e_j$ (resp. $e_1e_2 ...$) such that for each $\ell = 1, \ldots, j - 1$ (resp. $\ell = 1, \ldots$) we have $r(e_\ell) = s(e_{\ell + 1})$.
Note that for each vertex~$v$ distinct from $v_0$ there is at least one path starting at~$v_0$ and arriving to~$v$.

An \textit{ordered Bratteli diagram} $(V,E,\geq)$ is a Bratteli diagram $(V,E)$ together with a partial order~$\geq$ on~$E$, so that two edges are comparable if and only if they arrive at the same vertex.
For each $j \ge 1$ and $v \in V_j$ the partial order~$\ge$ induces an order on the set of paths from~$v_0$ to~$V$ as follows:
$$
e_1\cdots e_j > f_1\cdots f_j
$$
if and only there exists $j_0 \in \{1,\cdots, j \}$ such that $e_{j_0} >
f_{j_0}$ and such that for each $\ell \in \{ j_0 + 1, \ldots, j \}$ we have $e_\ell = f_\ell$.

We will say that an edge~$e$ is \textit{maximal} (resp. \textit{minimal}) if it is maximal (resp. minimal) with respect to the order~$\ge$ on the set of all edges in~$E$ arriving at~$r(e)$.
Note that for each vertex~$v$ distinct from~$v_0$ there is precisely one path starting at~$v_0$ and arriving to~$v$ that is maximal (resp. minimal) with respect to the order~$\ge$.
It is characterized as the unique path starting at~$v_0$ and arriving at~$v$ consisting of maximal (resp. minimal) edges.

\subsection{Bratteli-Vershik system}\label{ss:Bratteli-Vershik}
Fix an ordered Bratteli diagram $B \=(V,E,\geq)$.
We denote by~$X_B$ set of all infinite paths in~$B$ starting at~$v_0$.
For a finite path $e_1 \ldots e_j$ starting at~$v_0$ we denote by $U(e_1 \ldots e_j)$ the subset of~$X_B$ of all infinite paths $e_1'e_2' \ldots$ such that for all $\ell = 1, \ldots, j$ we have $e_\ell' = e_\ell$.
We endow~$X_B$ with the topology generated by the sets $U(e_1 \ldots e_j)$.
Then each of this sets is clopen, so $X_B$ becomes a compact Hausdorff space with a countable basis of clopen sets.

We will denote by~$X_B^{\max}$ (resp. $X_B^{\min}$) the set of all elements $(e_j)_{j \ge 1}$ of~$X_B$ so that for each $j \ge 1$ the edge~$e_j$ is a maximal (resp. minimal).
It is easy to see that each of these sets is non-empty.

From now on we assume that the set~$X_B^{\min}$ is reduced to a unique point, that we will denote by~$x_{\min}$.
We will then define the transformation $V_B:X_B\to X_B$ as follows:
\begin{itemize}
\item
$V_B^{-1}(x_{\min}) = X_{\max}$.
\item
Given $x \in X_B \setminus X_{\max}$, let $j \ge 1$ be the smallest integer such that~$e_j$ is not maximal.
Then we denote by~$f_j$ the successor of~$e_j$ and by $f_1\ldots f_{j - 1}$ the unique minimal path starting at~$v_0$ and arriving to~$s(f_k)$.
Then we put,
$$V_B(x)=f_1\cdots f_{k-1}f_ke_{k+1}e_{k+2}\ldots \ .$$
\end{itemize}
The map~$V_B$ is continuous, onto and invertible except at~$x_{\min}$.

For $j \ge 1$ and $v \in V_j$ we denote by~$s_j(v) > 0$ the number of paths starting at~$v_0$ and arriving to~$v$, and put $\vec{s}_j \= ( s_j(v) )_{v \in V_j} \in \R^{V_j}$.
Let $N_j \in \cM_{V_{j - 1}, V_j}$ be the matrix such that for each $v \in V_{j - 1}$ and $v' \in V_{j}$ the entry $N_j(v, v')$ is equal to the number of edges starting at~$v$ and arriving to~$v'$.
Observe that $N_j^t \vec{s}_{j - 1} = \vec{s}_{j}$, so if we put $B_0 = \{ 1 \} \in \cM_{V_0, V_0}$ and for each $j \ge 1$ we denote by $B_j \in \cM_{V_j, V_j}$ the diagonal matrix defined by $B_j(v, v) = s_j(v)$, then the matrix
$$ M_j \= B_{j - 1} N_j B_j^{-1} \in \cM_{V_{j - 1}, V_{j}}, $$
is stochastic.

The following result is well-known, see~\cite[Lemma~14]{CorRiv0804} for a proof in the precise setting considered here.
Recall that for a finite set $V$ we denote by~$\Delta_V$ the unit simplex in~$\R^V$.
\begin{lemm}\label{l:medidas Bratteli}
The space of probability measures on~$X_B$ that are invariant by~$V_B$, endowed with the weak$^*$ topology, is affine homeomorphic to $\varprojlim_{j} (\Delta_{V_j}, M_j)$.
\end{lemm}
\subsection{The Bratteli-Vershik system associated to a kneading map}\label{ss:kneading Bratteli-Vershik}
Given a kneading map~$Q$ we will now define an ordered Bratteli diagram $B_Q \= (V, E, \le)$ that was introduced by Bruin in~\cite[\S4]{Bru03}.

We start defining the Bratteli diagram $(V, E)$:
\begin{itemize}
\item
$V_0 = \{ 0 \}$, $V_1 = \{ k \in \N \mid Q(k) = 0 \}$ and for $j \ge 2$,
$$ V_j \= \{ k \in \N \mid k \ge j, Q(k-1) \le j - 2 \}. $$
\item
For $j \ge 1$,
\begin{multline*}
E_j = \{ j - 1 \to j \} \cup \{ j - 1 \to k \mid k \in V_j \setminus V_{j - 1} \}
\cup \\ \cup
\{ k \to k \mid k \in V_j \cap V_{j - 1} \}.
\end{multline*}
\end{itemize}
Note that for every $j \ge 2$, each vertex in~$V_j$ different from~$j$ has at most one edge arriving at it.
Besides $\{j - 1 \to j \} \in E_j$, the only edge that can arrive to $j \in V_j$ is $\{j \to j \} \in E_j$, that only exists when $j \in V_{j - 1}$.

So to define the partial order~$\ge$, we just have to define it, for each $j \ge 2$, between $\{j - 1 \to j \} \in E_{j - 1}$ and $\{ j \to j \} \in E_{j - 1}$ when both exist: we put $\{ j - 1 \to j \} < \{ j \to j \}$.
The rest of the edges are maximal and minimal at the same time.

Note that for $k \ge 1$ the set~$V_k$ is reduced to a point if and only if~$Q(k) = k - 1$.
So, if for each large $k \ge 1$ we have $Q(k) = k - 1$, then the set~$X_{B_Q}$ is finite.
Otherwise, it follows that the set~$X_{B_Q}$ is a Cantor set.

It is straight forward to check that the infinite path $0 \to 1 \to 2 \to \cdots $ is the unique minimal path in~$B_Q$.
Therefore there is a well defined map $V_{B_Q} : X_{B_Q} \to X_{B_Q}$, see~\S\ref{ss:Bratteli-Vershik}.
The following is \cite[Proposition~2]{Bru03}, and the last statement follows from \cite[Lemma~2]{BruKelsPi97}.
\begin{theo}[\cite{Bru03}, Proposition~2]\label{t:reduction to Bratteli}
Let~$Q$ be a kneading map that diverges to~$+\infty$, and consider the corresponding Bratteli-Vershik system $(X_{B_Q}, V_{B_Q})$ and generalized odometer $(\Omega_Q, T_Q)$.
Then there is a homeomorphism between~$X_{B_Q}$ and~$\Omega_Q$ that conjugates the action of~$V_{B_Q}$ on $X_{B_Q}$ to the action of~$T_Q$ on~$\Omega_Q$.
In particular $(X_{B_Q}, V_{B_Q})$ is minimal.
\end{theo}
We will also need the following lemma.
\begin{lemm}\label{l:recurrence on paths}
Let~$Q$ be a kneading map such that for every $k \in \N_0$ we have $Q(k) \le \max \{ 0, k - 2 \}$, and such that $Q(k) \to + \infty$ as $k \to + \infty$.
Let $(S_k)_{k \ge 1}$ be the sequence defined recursively by $S_0 = 1$ and $S_k = S_{k - 1} + S_{Q(k)}$.
Then for every $j \in \N$ we have $j + 1 \in V_j$, $s_j(j) = S_{j - 1}$, and for every $k \in V_j \setminus \{ j \}$ we have $s_j(k) = S_{Q(k - 1)}$.
\end{lemm}
\begin{proof}
That $j + 1$ is a direct consequence of the definition of~$V_j$ and the hypothesis that $Q(j) \le \max \{0, j - 2 \}$.
When $j = 1$, we have $Q(1) = 0$ and for all $k \in V_1$ we have $s_1(k) = S_0 = 1$.
So the assertions are satisfied in this case.
Suppose by induction that the assertions of the lemma hold for some $j \ge 1$.
Then by we have,
$$s_{j + 1}(j + 1) = s_j(j) + s_j(j + 1) = S_{j - 1} + S_{Q(j)} = S_j. $$
On the other hand, for each $k \in V_{j + 1} \setminus \{ j + 1 \}$ contained in~$V_j$ we have $s_{j - 1}(k) = s_j(k) = S_{Q(k - 1)}$.
Finally, for each $k \in \in V_{j + 1} \setminus \{ j + 1 \}$ not in~$V_j$ we have $Q(k - 1) = j - 1$ and
$$ s_{j + 1}(k) = s_j(j) = S_{j - 1} = S_{Q(k - 1)}.$$
\end{proof}


\section{Computing the space of invariant measures}\label{s:computing invariant}
This section is devoted to the proof of Theorem~\ref{t:invariant of doubly resonant}.
We start by showing that for each unimodal map whose kneading map is as in~\S\ref{ss:doubly resonant kneading}, the space of invariant probability measures supported on its \pcs{} is affine homeomorphic to the space of invariant probability measures of the corresponding generalized odometer.
In order to describe this space we calculate the transition matrices associated to the corresponding Bratteli-Vershik system.
The key calculation of a suitable product of these transition matrices is stated as Proposition~\ref{p:full block} in~\S\ref{ss:transition matrices}.
The proof of Theorem~\ref{t:invariant of doubly resonant} is given in~\S\ref{ss:proof of invariant of doubly resonant}.

\subsection{From the generalized odometer to the \pcs}\label{ss:measure isomorphism}
The purpose of this section is to prove the following proposition, whose proof is similar to that of~\cite[Theorem~B]{CorRiv0804}.
\begin{prop}\label{p:measure isomorphism}
Let~$\kneading$ be a kneading map defined as in~\S\ref{ss:doubly resonant kneading}, and let $(\Omega_{\kneading}, T_{\kneading})$ be the corresponding generalized odometer.
Let~$f$ be a unimodal map whose kneading map is equal to~$\kneading$, and denote by~$X_f$ its \pcs.
Then the space of invariant probability measures of~$(X_f, f|_{X_f})$ is affine homeomorphic to that of~$(\Omega_{\kneading}, T_{\kneading})$.
\end{prop}

The following lemma is similar to~\cite[Lemma~11]{CorRiv0804}.
Observe that, since~$T_Q^{-1}$ is well defined on~$\Omega_Q \setminus \{ \expansion{0} \}$, if we denote by $\mathcal{O}(\expansion{0})$
the grand orbit of $\expansion{0}$, then
$$ T^{-1}(\Omega_Q \setminus \mathcal{O}(\expansion{0}))
= \Omega_Q \setminus \mathcal{O}(\expansion{0}), $$ and all
negative iterates of~$T_Q$ are well defined on $\Omega_Q \setminus
\mathcal{O}(\expansion{0})$.
\begin{lemm}\label{l:criterion satisfied}
Let~$Q = \kneading$ be a kneading map as in~\S\ref{ss:doubly resonant kneading}.
Let $(\Omega_{Q}, T_{Q})$ be the corresponding generalized odometer.
Then for each constant $K > 0$, and for every pair of distinct points $x, x'$ in~$\Omega_{Q}$ that are not in the grand orbit of~$\expansion{0}$, there is an
integer~$m$ satisfying
$$ \max \{ q(T_{Q}^m(x)), q(T_{Q}^m(x')) \} \ge K
\text{ and }
{Q}(q(T_{Q}^m(x)) + 1) \neq {Q}(q(T_{Q}^m(x')) + 1). $$
\end{lemm}
\begin{proof}
Let $K \ge q_{2}$.

\partn{1} As in the proof of \cite[Lemma~10]{CorRiv0804}, it can be shown
that  there is an integer~$m'$ such that
$$ \max \{ q(T_{Q}^{m'}(x)), q(T_{Q}^{m'}(x')) \}
\ge K \text{ and } q(T_{Q}^{m'}(x)) \neq q(T_{Q}^{m'}(x')). $$

\partn{2}
Let $m'$ be the integer given by part~1, and put $y = (y_k)_{k \in \N_0} \= T_{Q}^{m'}(x)$ and $y' = (y'_k)_{k \in \N_0} \= T_{Q}^{m'}(x')$.
Assume without loss of generality that $q(y) < q(y')$, so that $q(y') \ge K \ge q_2$.
If ${Q}(q(y) + 1) \neq {Q}(q(y') + 1)$ then take $m = m'$.
So we assume that ${Q}(q(y) + 1) = {Q}(q(y') + 1)$.
Since~$q(y) \ge q_2$ we have ${Q}(q(y) + 1) = {Q}(q(y') + 1) \ge 1$, so there is~$n \in \N$ such that $q(y) + 1$ and $q(y') + 1$ belong to~$\lefti_n \cup \righti_n$.
The definition of~${Q}$ and of~$\Omega_{Q}$ imply that for each $k \in [q(y) + 1, q_{r_{n+1}-1} - 1]$ we have $y_k = 0$.
Indeed, suppose by contradiction that for such a~$k$ we have $y_k = 1$.
Then the definition of~$\Omega_{Q}$ implies that for every $j \in [{Q}(k+1), k - 1]$ we have~$y_j = 0$.
Since $k+1\in \lefti_n \cup \righti_n$, we get ${Q}(k + 1) \le q_{r_{n}-1} \le q(y)$,  which is a contradiction.

Therefore $( \hat{y}_k )_{k \ge 0} \= T_{Q}^{-S_{q(y)}}(y)$ is such
that for all $k \in [0, q_{r_{n+1}-1} - 1]$ we have $\hat{y}_k = 0$.
Since~$y$ is not in the grand orbit of~$\expansion{0}$ this implies that $q(T_{Q}^{-S_{q(y)}}(y)) \ge q_{r_{n+1}-1} \ge q(y') \geq K $.
Thus, from the definition of ${Q}$ we have
$${Q}(q(T_{Q}^{-S_{q(y)}}(y)) + 1) \geq q_{r_{n}}.$$
On the other hand, since $q(y) < q(y')$, we have
$q(T_{Q}^{-S_{q(y)}}(y')) \le q(y') - 1\leq q_{r_{n+1}-1} - 2$, so
$$ {Q}(q(T_{Q}^{-S_{q(y)}}(y')) + 1) \le q_{r_{n} - 1}. $$
This shows that the integer $m = m' - S_{q(y)}$ satisfies the
desired properties.
\end{proof}
\begin{proof}[Proof of Proposition~\ref{p:measure isomorphism}]
Since the logistic family is full there is a parameter $\parameter \in (0, 4]$ such that the kneading map of the logistic map~$f_\parameter$ is~$\kneading$.
Denote by~$X_{f_{\parameter}}$ the \pcs{} of~$f_\parameter$.
By Proposition~\ref{p:continuous unimodal} the spaces of invariant measures of~$(f, X_f)$ and~$(f_{\parameter}, X_{f_{\parameter}})$ are affine homeomorphic.
So, without loss of generality we assume that~$f$ is a logistic map.
This  ensures the existence of the factor map $\pi : \Omega_{\kneading} \to X_{f}$ defined above.
Since for every sufficiently large integer~$k$ inequality~\eqref{e:almost admissibility} is satisfied, \cite[Lemma~11]{CorRiv0804} implies there is a constant $K>0$ such that for every pair of distinct points $x,x'$ in $\Omega_{\kneading}$ that are not in the grand orbit of $\expansion{0}$ and that satisfy
$$ \max\{q(x), q(x'))\} \ge K \text{ and } {\kneading}(q(x) + 1) \neq {\kneading}(q(x') + 1), $$
we have $\pi(x)\neq \pi(x')$.
Thus, from Lemma~\ref{l:criterion satisfied} we deduce that~$\pi$ is injective on $\Omega_{\kneading} \setminus \cO(\expansion{0})$.
The rest of the proof follows as the proof of~\cite[Proposition~9]{CorRiv0804}.
\end{proof}
\subsection{Transition matrices}\label{ss:transition matrices}
For each $n \in \N$ let $\vec{a}_n \in \N^{[0, n]}$ be given, and put $a \= (\vec{a}_n)_{n \in \N}$.
Furthermore, let $q \= (q_r)_{r \in \N_0}$ be a strictly increasing sequence of integers such that~$q_0 = 0$ and such that for each $n \in \N$ we have
$$ q_{r_n} - q_{r_n - 1} = a_{n, 0} + \cdots + a_{n, n}, $$
and let $Q \= \kneading$ be the corresponding kneading map defined in~\S\ref{ss:doubly resonant kneading}.

Let $B_Q = (V, E, \le)$ be the ordered Bratteli-Vershik diagram associated to the kneading map~$Q$.
From the definition of~$B_Q$ it follows that $V_1 = Q^{-1}(0) \setminus \{ 0 \} = [1, q_2]$, and that for each $j \in [2, q_1 + 1]$ we have $V_j = [j, q_2 + 1]$.
Furthermore, for each $n \in \N_0$, $m \in [0, n]$, and $j \in [q_{r_n + m} + 2, q_{r_n + m + 1} + 1]$ we have
\begin{equation}\label{e:generic cut 1}
 V_j = [j, q_{r_{n + 1} - 1} + 1] \cup \left( 1 + \bigcup_{i = 0}^m \left( \lefti_{n + 1, i} \cup \righti_{n + 1, i} \right) \right),
\end{equation}
and that when $m = n + 1$ we have $r_n + n + 1 = r_{n + 1} - 1$, and for $j \in [q_{r_{n + 1} - 1} + 2, q_{r_{n + 1}} + 1]$ we have
\begin{equation}\label{e:generic cut 2}
V_j = [j, q_{r_{n + 2} - 1} + 1].
\end{equation}
Note in particular that for every $n \in \N$ we have
\begin{equation}\label{e:previous last cut}
V_{q_{r_n - 1} + 1} = [q_{r_n - 1} + 1, q_{r_{n + 1} - 1} + 1] \setminus \left( 1 + (\lefti_{n, n} \cup \righti_{n, n}) \right),
\end{equation}
\begin{equation}\label{e:last cut}
V_{q_{r_n} + 1} = [q_{r_n} + 1, q_{r_{n + 1} - 1} + 1].
\end{equation}

\begin{prop}\label{p:full block}
Given $n \in \N$, for each $k \in [q_{r_n}, q_{r_{n + 1} - 1}]$ put
$$ \vec{v}(k) \=  \frac{S_{q_{r_n}}}{S_k} \vec{e}_{q_{r_n} + 1}
+
\sum_{i = q_{r_n} + 1}^k \frac{S_{Q(i)}}{S_k} \vec{e}_{i + 1} \in \R^{V_{q_{r_n} + 1}}.$$
Then the columns of the matrix
$$ M_{q_{r_n} + 2} \cdots M_{q_{r_{n + 1}} + 1} \in \mathcal{M}([q_{r_{n}} + 1, q_{r_{n + 1} - 1} + 1], [q_{r_{n + 1}} + 1, q_{r_{n + 2} - 1} + 1]), $$
are given by,
\begin{multline*}
M_{q_{r_n} + 2} \cdots M_{q_{r_{n + 1}} + 1} ( \cdot, q_{r_{n + 1}} + 1)
= \\ =
\frac{S_{q_{r_{n + 1} - 1}}}{S_{q_{r_{n + 1}}}}\vec{v}(q_{r_{n + 1} - 1}) + \sum_{m = 0}^{n + 1} \frac{S_{q_{r_n + m}}}{S_{q_{r_{n + 1}}}} a_{n + 1, m} \vec{v}(q_{r_n + m}),
\end{multline*}
and for every $m \in [0, n + 1]$ and $\ell \in V_{q_{r_{n + 1}} + 1}$ such that $\ell - 1 \in J_{n + 1, m}$, by
\[
M_{q_{r_n} + 2} \cdots M_{q_{r_{n + 1}} + 1} ( \cdot, \ell)
= 
\vec{v}(q_{r_n + m}).
\]
In particular the rank of the matrix $M_{q_{r_n} + 2} \cdots M_{q_{r_{n + 1}} + 1}$
 is equal to~$n + 2$.
\end{prop}
The proof of this proposition depends on the following lemma.
\begin{lemm}\label{l:basic blocks}
Fix $n \in \N$ and let $m \in [0, n + 1]$.
If $m \in [0, n]$ then the set $V_{q_{r_n + m + 1} + 1}$ is equal to the disjoint union of $V_{q_{r_n + m} + 1} \setminus [q_{r_n + m} + 1, q_{r_n + m + 1}]$ and of $1 + Q^{-1}(q_{r_n + m})$, and we have
\begin{multline}\label{e:non-final basic block}
N_{q_{r_n + m} + 2} \cdots N_{q_{r_n + m + 1} + 1} (\cdot, \ell)
= \\ =
\begin{cases}
\sum_{k = q_{r_n + m}}^{q_{r_n + m + 1}} \vec{e}_{k + 1} & \text{if } \ell = q_{r_n + m + 1} + 1; \\
\vec{e}_\ell & \text{if } \ell \in V_{q_{r_n + m} + 1} \setminus [q_{r_n + m} + 1, q_{r_n + m + 1} + 1]; \\
\vec{e}_{q_{r_n + m} + 1} & \text{if } \ell - 1 \in Q^{-1}(q_{r_n + m}).
\end{cases}
\end{multline}
When $m = n + 1$ we have $r_n + m = r_{n + 1} - 1$, $V_{q_{r_{n + 1}} + 1} = [q_{r_{n + 1}} + 1, q_{r_{n + 1} - 1} + 1]$,  and
\begin{multline}\label{e:final basic block}
N_{q_{r_{n + 1} -1} + 2} \cdots N_{q_{r_{n + 1}} + 1} (\cdot, \ell)
= \\ =
\begin{cases}
a_{n + 1, n + 1} \vec{e}_{q_{r_{n + 1} - 1} + 1} + \sum_{k = q_{r_{n + 1} - 1}}^{q_{r_{n + 1}} - a_{n + 1, n + 1}}\vec{e}_{k + 1}
& \text{if } \ell = q_{r_{n + 1}} + 1; \\
\vec{e}_\ell & \text{if } \ell - 1 \in \righti_{n + 1} \setminus \righti_{n + 1, n + 1}; \\
\vec{e}_{q_{r_{n + 1} - 1} + 1} & \text{if } \ell - 1 \in \righti_{n + 1, n + 1}.
\end{cases}
\end{multline}
\end{lemm}
\begin{proof}
It follows from the definition of~$B_Q$ that for each $r \in \N$ the set $V_{q_r + 2}$ is equal to the disjoint union of $V_{q_r + 1} \setminus \{ q_r + 1 \}$ and $1 + Q^{-1}(q_r)$, and that
$$ N_{q_r + 2} (\cdot, \ell) =
\begin{cases}
\vec{e}_{q_r + 1} + \vec{e}_{q_r + 2} & \text{if } \ell = q_r + 2; \\
\vec{e}_{\ell} & \text{if } \ell \in V_{q_r + 1} \setminus \{ q_r + 1, q_r + 2 \}; \\
\vec{e}_{q_r + 1} & \text{if } \ell - 1 \in Q^{-1}(q_r).
\end{cases}
$$
On the other hand, for each $j \in [q_r + 3, q_{r + 1} + 1]$ we have $V_{j - 1} = V_j \cup \{ j - 1 \}$ and
$$ N_j ( \cdot, \ell) =
\begin{cases}
\vec{e}_{j - 1} + \vec{e}_j & \text{if } \ell = j; \\
\vec{e}_\ell & \text{if } \ell \in V_j \setminus \{ j \}.
\end{cases}
$$
A direct computation using the fact that for every $j \in [q_r + 2, q_{r + 1}]$ the set~$V_{j}$ is the disjoint union of $V_{q_{r + 1} + 1}$ and $[j, q_{r + 1}]$, shows that
\begin{multline}\label{e:generic block}
N_{q_r + 2} \cdots N_{q_{r + 1} + 1} (\cdot, \ell)
= \\ =
\begin{cases}
\sum_{j = q_r + 2}^{q_{r + 1} + 1} N_{q_r + 2}(\cdot, j) & \text{if } \ell = q_{r + 1} + 1; \\
N_{q_r + 2}(\cdot, \ell) & \text{if } \ell \in V_{q_{r} + 2} \setminus [q_{r} + 2, q_{r + 1} + 1].
\end{cases}
\end{multline}

Fix $n \in \N$ and let $m \in [0, n + 1]$.

When $m \in [0, n]$ the assertion of the lemma about~$V_{q_{r_n + m + 1} + 1}$ follows from~\eqref{e:generic cut 1}.
On the other hand, by~\eqref{e:generic cut 2} if $m = 0$ and by~\eqref{e:generic cut 1} if $m \in [1, n]$, we have
$$ [q_{r_n + m} + 2, q_{r_n + m + 1} + 1]
\subset
V_{q_{r_n + m} + 1}. $$
Then in this case the assertion of the lemma follows easily from~\eqref{e:generic block} with $r = r_n + m$.

Suppose now that $m = n + 1$.
By the definition of the sequence $(r_{n'})_{n' \in \N_0}$ we have $r_n + n + 1 = r_{n + 1} - 1$.
The assertion about $V_{q_{r_{n + 1}} + 1}$ is given by~\eqref{e:last cut} with~$n$ replaced by~$n + 1$.
Since by~\eqref{e:previous last cut} we have
\begin{multline*}
[q_{r_{n + 1} -1} + 2, q_{r_{n + 1}} + 1] \cap V_{q_{r_{n + 1} - 1} + 1}
=
1 + (\lefti_{n + 1} \setminus \lefti_{n + 1, n + 1})
= \\ =
[q_{r_{n + 1} -1} + 2, q_{r_{n + 1}} + 1 - a_{n + 1, n + 1}]
\end{multline*}
and since $\lefti_{n + 1, n + 1} \subset Q^{-1}(q_{r_{n + 1} - 1})$, we conclude from~\eqref{e:generic block} with $r = r_{n + 1} - 1$, that
\begin{multline*}
N_{q_{r_{n + 1} - 1} + 2} \cdots N_{q_{r_{n + 1}} + 1} (\cdot, q_{r_{n + 1}} + 1)
= \\ =
a_{n + 1, n + 1} \vec{e}_{q_{r_{n + 1} - 1} + 1} + \sum_{k = q_{r_{n + 1} - 1}}^{q_{r_{n + 1}} - a_{n + 1, n + 1}}\vec{e}_{k + 1}.
\end{multline*}
On the other hand, for
$$ \ell \in [q_{r_{n + 1}} + 2, q_{r_{n + 2} - 1} + 1] \cap V_{q_{r_{n + 1} - 1} + 1}
=
1 + \left( \righti_{n + 1} \setminus \righti_{n + 1, n + 1} \right), 
$$
we have
$$ N_{q_{r_{n + 1} - 1} + 2} \cdots N_{q_{r_{n + 1}} + 1} (\cdot, \ell)
=
\vec{e}_\ell, $$
and that for
$$ \ell \in [q_{r_{n + 1}} + 2, q_{r_{n + 2} - 1} + 1] \cap \left( 1 + Q^{-1}(q_{r_{n + 1} - 1}) \right)
=
1 + \righti_{n + 1, n + 1},
$$ we have
$$ N_{q_{r_{n + 1} - 1} + 2} \cdots N_{q_{r_{n + 1}} + 1} (\cdot, \ell)
=
\vec{e}_{q_{r_{n + 1} - 1} + 1}. $$
This completes the proof of the lemma.
\end{proof}

\begin{proof}[Proof of Proposition~\ref{p:full block}]
Fix $n \in \N$.

\partn{1}
We will show by induction that for each $m_0 \in [1, n + 1]$ we have
\begin{multline}\label{e:almost full block}
N_{q_{r_n} + 2} \cdots N_{q_{r_n + m_0} + 1} ( \cdot, \ell)
= \\ =
\begin{cases}
\sum_{k = q_{r_n}}^{q_{r_n + m_0}} \vec{e}_{k + 1} & \text{if } \ell = q_{r_n + m_0} + 1; \\
\vec{e}_\ell & \text{if } \ell \in [q_{r_n + m_0} + 2, q_{r_{n + 1} - 1} + 1]; \\
\sum_{k = q_{r_n}}^{q_{r_n + m}} \vec{e}_{k + 1} & \text{if } \ell - 1 \in \lefti_{n + 1, m} \cup \righti_{n + 1, m} \text{ and } m \in [0, m_0 - 1].
\end{cases}
\end{multline}
The case $m_0 = 1$ is given by~\eqref{e:non-final basic block} with $m =  0$.
Suppose that this holds for some $m_0 \in [1, n]$.
Observe that by~\eqref{e:generic cut 1} the set~$V_{q_{r_n + m_0 + 1} + 1}$ is the disjoint union of
$$ V_{q_{r_n + m_0} + 1} \setminus [q_{r_n + m_0} + 1, q_{r_n + m_0 + 1}] $$
and
$$ 1 + Q^{-1}(q_{r_n + m_0 + 1}) = 1 + \lefti_{n + 1, m_0 + 1} \cup \righti_{n + 1, m_0 + 1}. $$
In view of~\eqref{e:non-final basic block} with $m = m_0$ we obtain
\begin{multline*}
N_{q_{r_n} + 2} \cdots N_{q_{r_n + m_0 + 1} + 1} ( \cdot, q_{r_n + m_0 + 1} + 1)
= \\ =
\sum_{k = q_{r_n + m_0}}^{q_{r_n + m_0 + 1}} N_{q_{r_n} + 2} \cdots N_{q_{r_n + m_0} + 1} (\cdot, k + 1)
=
\sum_{k = q_{r_n}}^{q_{r_n + m_0 + 1}} \vec{e}_{k + 1},
\end{multline*}
for each $\ell \in V_{q_{r_n + m_0} + 1} \setminus [q_{r_n + m_0} + 1, q_{r_n + m_0 + 1} + 1]$ we have
\begin{multline*}
N_{q_{r_n} + 2} \cdots N_{q_{r_n + m_0 + 1} + 1} ( \cdot, \ell)
=
N_{q_{r_n} + 2} \cdots N_{q_{r_n + m_0} + 1} ( \cdot, \ell)
= \\ =
\begin{cases}
\vec{e}_{\ell} & \text{if } \ell \in [q_{r_n + m_0 + 1} + 2, q_{r_{n + 1} - 1} + 1] \\
\sum_{k = q_{r_n}}^{q_{r_n + m}} \vec{e}_{k + 1} & \text{if } \ell - 1 \in \lefti_{n + 1, m} \cup \righti_{n + 1, m} \text{ and } m \in [0, m_0 - 1].
\end{cases}
\end{multline*}
Finally, for each~$\ell \in V_{q_{r_n + m_0 + 1} + 1}$ such that $\ell - 1 \in \lefti_{n + 1, m_0 + 1} \cup \righti_{n + 1, m_0 + 1}$ we have
\begin{multline*}
N_{q_{r_n} + 2} \cdots N_{q_{r_n + m_0 + 1}} ( \cdot, \ell)
= \\ =
N_{q_{r_n} + 2} \cdots N_{q_{r_n + m_0}} ( \cdot, q_{r_n + m_0} + 1)
=
\sum_{k = q_{r_n}}^{q_{r_n + m_0}} \vec{e}_{k + 1}.
\end{multline*}
This completes the proof of the induction step.

\partn{2}
A direct computation using~\eqref{e:almost full block} with $m_0 = n$ and~\eqref{e:final basic block}, gives
\begin{multline*}
N_{q_{r_n} + 2} \cdots N_{q_{r_{n + 1}} + 1} (\cdot, q_{r_{n + 1}} + 1)
= \\ =
\sum_{k = q_{r_{n}}}^{q_{r_{n + 1} - 1}} \vec{e}_{k + 1} + \sum_{m = 0}^{n + 1} a_{n + 1, m} \sum_{k = q_{r_n}}^{q_{r_n + m}} \vec{e}_{k + 1},
\end{multline*}
and for every $m \in [0, n + 1]$ and $\ell \in V_{q_{r_{n + 1}} + 1}$ such that $\ell - 1 \in \righti_{n + 1, m}$,
\begin{equation*}
N_{q_{r_n} + 2} \cdots N_{q_{r_{n + 1}} + 1} (\cdot, \ell)
=
\sum_{k = q_{r_n}}^{q_{r_n + m}} \vec{e}_{k + 1}.
\end{equation*}
The assertion of the proposition is then a direct consequence of the definition of the matrices~$M_j$ and Lemma~\ref{l:recurrence on paths}.

\end{proof}

\subsection{Proof of Theorem~\ref{t:invariant of doubly resonant}}\label{ss:proof of invariant of doubly resonant}
Let~$f$ be a unimodal having~$Q$ as kneading map.
That the \pcs{} of~$f$ is a Cantor set and that~$f$ is minimal on this set is given by Proposition~\ref{p:continuous unimodal}.
In view of Proposition~\ref{p:measure isomorphism} and Theorem~\ref{t:reduction to Bratteli},  it is enough to prove that the space of invariant probability measures of the Bratteli-Vershik system $(X_{B_Q}, V_{B_Q})$ is affine homeomorphic to~$\varprojlim_{r} (\Delta_{[0, n + 1]}, \Xi_n)$.

For each $n \in \N$ let $\Pi_n : \R^{V_{q_{r_n} + 1}} \to \R^{[0, n + 1]}$ be the stochastic matrix defined by
$$ \Pi_n (x_{q_{r_n} + 1}, \ldots, x_{q_{r_{n + 1} - 1} + 1})
=
\left( \left( \sum_{k = q_{r_n + n} + 2}^{q_{r_n + n + 1} + 1} x_k \right), \ldots, \left( \sum_{k = q_{r_n} + 2}^{q_{r_n + 1} + 1} x_k \right), x_{q_{r_n} + 1} \right). $$
Using the definition of~$\vec{v}(k)$ in the statement of Proposition~\ref{p:full block}, for each $m_0 \in [0, n + 1]$ we put
\begin{multline*}
\vec{w}_n(m_0) \= \Pi_n (\vec{v}(q_{r_n + n + 1 - m_0}))
= \\ = 
\sum_{m = m_0}^{n} \frac{S_{q_{r_n + n + 1 - m}} - S_{q_{r_n + n - m}}}{S_{q_{r_n + n + 1 - m_0}}} \vec{e}_{m}
+
\frac{S_{q_{r_n}}}{S_{q_{r_n + n + 1 - m_0}}} \vec{e}_{n + 1}.
\end{multline*}
Furthermore, when $n \ge 2$, we denote by $A_n : \R^{[0, n + 1]} \to \R^{[0, n]}$ the stochastic matrix defined for $m \in [0, n]$ by $A_n(\cdot, m) = \vec{w}_{n - 1}(m)$ and by
\begin{equation*}
A_n( \cdot, n + 1)
=
\frac{S_{q_{r_{n} - 1}}}{S_{q_{r_{n}}}}\vec{w}_{n - 1}(0) + \sum_{m = 0}^{n} \frac{S_{q_{r_{n - 1} + n - m}}}{S_{q_{r_{n}}}} a_{n, n - m} \vec{w}_{n - 1}(m).
\end{equation*}

A direct computation shows that for each $n \ge 2$ we have
$$ \Pi_{n - 1} M_{q_{r_{n - 1} + 2}} \cdots M_{q_{r_{n}} + 1} = A_n \Pi_{n}. $$
Therefore the sequence of maps $(\Pi_n)_{n \ge 1}$ define a continuous linear map
$$ \Pi : \varprojlim_j (\R^{V_j}, M_j) \to \varprojlim_n (\R^{[0, n + 1]}, A_n), $$
mapping $\varprojlim_j (\R^{V_j}, M_j)$ onto $\varprojlim_n (\Delta_{[0, n + 1]}, A_n)$.
By Proposition~\ref{p:full block} the rank of the matrix $M_{q_{r_{n - 1} + 2}} \cdots M_{q_{r_{n}} + 1}$ is equal to~$n + 1$, so~$\Pi$ is a homeomorphism and the inverse limits $\varprojlim_j (\R^{V_j}, M_j)$ and $\varprojlim_n (\Delta_{[0, n + 1]}, A_n)$ are affine homeomorphic.

In view of Lemma~\ref{l:equivalence}, the following lemma together with the hypothesis
$$ \sum_{r \in \N \setminus \{ r_n \mid n \in \N \}} \frac{S_{q_{r - 1}}}{S_{q_{r}}} < + \infty, $$
imply that the inverse limit $\varprojlim_n (\Delta_{[0, n + 1]}, A_n)$ is affine homeomorphic to $\varprojlim_n (\Delta_{[0, n + 1]}, \Xi_n)$.
This completes the proof of Theorem~\ref{t:invariant of doubly resonant}.

\begin{lemm}
For each $n \in \N$ and $m_0 \in [0, n - 1]$ we have
$$ \| A_n(\cdot, m_0) - \Xi_n(\cdot, m_0) \|_1
=
2 \frac{S_{q_{r_{n - 1} + n - m_0 - 1}}}{S_{q_{r_{n - 1} + n - m_0}}}.$$
Furthermore $A_n(\cdot, n) = \Xi_n(\cdot, n)$, and 
$$ \| A_n(\cdot, n + 1) - \Xi_n(\cdot, n + 1) \|_1
\le
2 \sum_{m = 1}^n \frac{S_{q_{r_{n - 1} + m - 1}}}{S_{q_{r_{n - 1} + m}}}. $$
\end{lemm}
\begin{proof}
By definition we have $A_n(\cdot, n) = \vec{w}_{n - 1}(n) = \vec{e}_n = \Xi_n(\cdot, n)$, and for each $m_0 \in [0, n - 1]$ we have
\begin{equation}\label{e:distance to canonical}
\begin{split}
\| A_n(\cdot, m_0) & - \Xi_n(\cdot, m_0) \|_1
\\ = &
\| \vec{w}_{n - 1}(m_0) - \vec{e}_{m_0} \|_1
\\ = &
\left| \frac{S_{q_{r_{n - 1} + n - m_0}} - S_{q_{r_{n - 1} + n - m_0 - 1}}}{S_{q_{r_{n - 1} + n - m_0}}}  - 1 \right|
\\ & +
\sum_{m = m_0 + 1}^{n - 1}  \frac{S_{q_{r_{n - 1} + n - m}} - S_{q_{r_{n - 1} + n - m - 1}}}{S_{q_{r_{n - 1} + n - m_0}}}
+
\frac{S_{q_{r_{n - 1}}}}{S_{q_{r_{n - 1}} + n - m_0}}
\\ = &
2 \frac{S_{q_{r_{n - 1} + n - m_0 - 1}}}{S_{q_{r_{n - 1} + n - m_0}}}.
\end{split}
\end{equation}
On the other hand,
$$
\| A_n(\cdot, n + 1) - \Xi_n(\cdot, n + 1) \|_1
\le
\sum_{m_0 = 1}^n \| \vec{w}_{n - 1}(m_0) - \vec{e}_{m_0} \|_1,
$$
so the final assertion follows from~\eqref{e:distance to canonical}.
\end{proof}

\subsection{Example}\label{ss:example}
Given~$\beta \in \R \setminus \Q$ put
$$ G(\beta) \= \Z + \beta\Z
\text{ and }
G_+(\beta)=\{m+\beta n\geq 0 \mid m,n\in\Z\}. $$
For each such~$\beta$ we will construct a kneading map~$Q$ such that the dimension group associated to the generalized odometer $(\Omega_Q, T_Q)$, and hence to its natural extension, is isomorphic to $(G(\beta), G_+(\beta), 1)$.
Thus we deduce that every simple dimension group which is free of rank $2$ is isomorphic as ordered group to the dimension group associated to (the natural extension of) a generalized odometer associated to a kneading map.
Since the rational subdimension group of~$(G(\beta), G_+(\beta), 1)$ is~$(\Z, \N_0, 1)$, it follows from~\cite[\S4.1]{GjeJoh00} that the dimension group~$(G(\beta), G_+(\beta), 1)$ is not isomorphic to the dimension group associated to a Toeplitz flow, nor to that of an odometer.

Before defining the kneading map~$Q$, note that the dimension groups $(G(\beta), G_+(\beta), 1)$, $(G(\beta + 1), G_+(\beta + 1), 1)$, and $(G(1 - \beta), G_+(1 - \beta), 1)$ are isomorphic to each other.
So we can restrict to the case where~$\beta \in (0, \tfrac{1}{2})$.
Let $k\geq 2$ be the integer determined by $\beta\in (\frac{1}{k+1}, \frac{1}{k})$, and let $[0,a_1,a_2,a_3,\cdots]$ be the continued fraction expansion of~$\alpha \= \tfrac{1}{\beta} - k \in [0, 1] \setminus \Q$.
Consider the function~$Q : \N_0 \to \N_0$ defined by
$$Q(l)=
\begin{cases}
0 & \mbox{ if } l \in [0, k]; \\
k-1 & \mbox{ if } l \in [k + 1, k + a_1]; \\
k-1+\sum_{i=1}^na_i & \mbox{ if } l \in \left[k+1 +\sum_{i=1}^na_i, k+\sum_{i=1}^{n+1}a_i \right]. \\
\end{cases}
$$
It is non-decreasing and such that for every~$l \ge 1$ we have $Q(l) \le l - 1$.
So~$Q$ is a kneading map.
If $(M_j)_{j \in \N}$ is the corresponding sequence of transition matrices, then it is easy to see that for each~$n \in \N$ we have
$$
A_n
\= 
M_{k + 1 + a_1 + \cdots + a_{n - 1}} \cdots M_{k + a_1+\cdots + a_{n}}
=\left(%
\begin{array}{cc}
  a_n & 1 \\
  1 & 0 \\
\end{array}%
\right).
$$
By considering a Bratteli diagram isomorphic to~$B_Q$, that only differs with it in the first~$k + 1$ levels, we obtain
$$
A_0 \= M_2 \cdots M_k = \left(%
\begin{array}{cc}
  k -1 & 1 \\
  1 & 0 \\
\end{array}%
\right),
$$
so the dimension group associated to~$(\Omega_Q, T_Q)$ is isomorphic to direct limit
$$ \Z \xrightarrow{(1, 1)^T} \Z^2 \xrightarrow{A_0} \Z^2 \xrightarrow{A_1} \Z^2 \xrightarrow{A_2} \cdots, $$
see for example~\cite[Theorem~3.7]{GioPutSka95}, which by~\cite[Theorem~4.8]{Eff81} is isomorphic to $(G(\beta), G_+(\beta),1)$.

\appendix
\section{Measures of zero Lyapunov exponent of complex maps}\label{a:ergodic theory}
The purpose of this appendix is to prove Corollary~\ref{c:ergodic theory holomorphic}.
As the parameters given by (the proof of) the Main Theorem are such that the kneading map of corresponding logistic map diverges to~$+ \infty$, this result is a direct consequence of the following lemma.
\begin{lemm}\label{l:ergodic theory holomorphic}
Let~$\parameter \in (0, 4]$ be a parameter such that the kneading map of~$f_\parameter$ diverges to~$+\infty$.
Consider the corresponding quadratic polynomial~$P_\parameter$, and denote by~$t_0$ the Hausdorff dimension of the Julia set of~$P_\parameter$.
Then, for an invariant probability measure of~$P_\parameter$ that is supported on the Julia set of~$P_\parameter$, the following properties are equivalent.
\begin{enumerate}
\item[1.]
It is supported on the \pcs{} of~$f_\parameter$.
\item[2.]
Its Lyapunov exponent is zero.
\item[3.]
It is an equilibrium state of~$P_\parameter$ for the potential~$- t_0 \log |P_\parameter'|$, whose Lyapunov exponent is zero.
\end{enumerate}
\end{lemm}
\begin{proof}
As the restriction of~$P_\parameter$ to~$[0, 1]$ is the logistic map~$f_\parameter$, the implication $1 \Rightarrow 2$ is given by \cite[Lemma~21]{CorRiv0804}.
The proof of the implication $2 \Rightarrow 1$ is analogous to the corresponding implication of the same lemma.

The implication $3 \Rightarrow 2$ being trivial we just need to prove the implication $2 \Rightarrow 3$ to complete the proof of the lemma.
We will show that for each invariant measure~$\mu$ that is supported on the Julia set of~$P_\parameter$, we have $h_\mu(P_\parameter) - t_0 \int \log |P_\parameter'| d \mu \le 0$ with equality when the Lyapunov exponent of~$\mu$ is zero.
For such a measure we have $\int \log |P_\parameter'| d \mu \ge 0$ by~\cite{Prz93}, so there are two cases.
If $\int \log |P_\parameter'| d \mu = 0$, then~$\mu$ is supported on the \pcs{} of~$P_\parameter$ and therefore we have $h_\mu(P_\parameter) = 0$ by the variational principle and the fact that the topological entropy of~$P_\parameter$ restricted to its \pcs{} is zero (Proposition~\ref{p:continuous unimodal}).
So we have $h_\mu(P_\parameter) - t_0 \int \log |P_\parameter'| d \mu = 0$ in this case.
Suppose now that $\int \log |P_\parameter'| d\mu > 0$.
Then by~\cite{Led84,Man88} it follows that, if we denote by $\mathrm{HD}(\mu)$ the Hausdorff dimension of~$\mu$, then $h_\mu(P_\parameter) = \mathrm{HD}(\mu) \int \log |P_\parameter'| d \mu$.
As $\mathrm{HD}(\mu) \le t_0$, we obtain $h_\mu(P_\parameter) - t_0 \int \log |P_\parameter'| d \mu \le 0$ in this case.
\end{proof}

\bibliographystyle{alpha}

\end{document}